\documentclass[11pt]{article}
\usepackage{graphicx,amsmath,color,amsthm,amssymb}
\usepackage{psfrag,picture}
\usepackage{tikz}
\usepackage{hyperref}
\usepackage{subfig}
\usepackage{soul}

\newcommand{\ba}{\begin{array}}
\newcommand{\ea}{\end{array}}
\newcommand{\be}{\begin{equation}}
\newcommand{\ee}{\end{equation}}
\newcommand{\ben}{\begin{equation*}}
\newcommand{\een}{\end{equation*}}
\newcommand{\bd}{\begin{displaymath}}
\newcommand{\ed}{\end{displaymath}}
\newcommand{\bi}{\begin{itemize}}
\newcommand{\ei}{\end{itemize}}
\newcommand{\bn}{\begin{enumerate}}
\newcommand{\en}{\end{enumerate}}

\newtheorem{lemma}{Lemma}
\newtheorem{theorem}{Theorem}

\newtheorem{remark}{Remark}
\newtheorem{assumption}{Assumption}

\title{Convergence of finite difference methods \\for the wave equation in two space dimensions}
\author{Siyang Wang\thanks{Division of Scientific Computing, Department
    of Information Technology, Uppsala University, SE-751 05 Uppsala,
    Sweden. \href{mailto:siyang.wang@it.uu.se}{Email: siyang.wang@it.uu.se}}, ~ Anna Nissen\thanks{Division of Numerical Analysis, Department of Mathematics, Royal Institute of Technology, SE-100 44, Stockholm, Sweden. \href{mailto:siyang.wang@it.uu.se}{}},~ and Gunilla Kreiss\thanks{Division of Scientific Computing, Department
    of Information Technology, Uppsala University, SE-751 05 Uppsala,
    Sweden. \href{mailto:siyang.wang@it.uu.se}{}}}
\begin{document}
\maketitle

\begin{abstract}
When using a finite difference method to solve an initial--boundary--value problem, the truncation error is often of lower order at a few grid points near boundaries than in the interior.  Normal mode analysis is a powerful tool to analyze the effect of the large truncation error near boundaries on the overall convergence rate, and has been used in many previous literatures for different equations. However, existing work only concerns problems in one space dimension. In this paper, we extend the analysis to problems in two space dimensions. The two dimensional analysis is based on a diagonalization procedure that decomposes a two dimensional problem to many one dimensional problems of the same type. We present a general framework of analyzing convergence for such one dimensional problems, and explain how to obtain the result for the corresponding two dimensional problem. In particular, we consider two kinds of truncation errors in two space dimensions: the truncation error along an entire boundary, and the truncation error localized at a few grid points close to a corner of the computational domain. The accuracy analysis is in a general framework, here applied to the second order wave equation. Numerical experiments corroborate our accuracy analysis. 
\end{abstract}


%
%
%

\section{Introduction}
Wave propagation problems can often be efficiently discretized with high order finite difference methods. Due to stability consideration, the formal accuracy order of the discretization scheme is typically considerably lower close to computational boundaries than that of the interior scheme. However, the numerical solution often converges at a rate higher than indicated by the boundary truncation error, a phenomenon termed as \emph{gain in convergence}. This phenomenon can partly be understood from the fact that the number of grid points with the lower order stencil is independent of grid spacing. Analysis is needed to determine the precise order of gain in convergence. 

There are two different methods for analyzing how much is gained in the convergence rate, the energy method and the normal mode analysis \cite{Gustafsson2008,Gustafsson2013}. Applying the energy method in a straightforward way indicates a half order gain in convergence. An exception is found in \cite{Abarbanel2000}, where a careful energy analysis performed to the heat equation gives a gain of one and a half orders in convergence. However, the computations show a gain of two orders. 

Sharp error estimates can be obtained by the normal mode analysis. In \cite{Svard2006}, it is noted by a normal mode approach that the gain in convergence can equal to the highest order of spatial derivatives in the equation. We refer to this gain as optimal. When a so--called determinant condition is satisfied, normal mode analysis can straightforwardly be used to prove that  the gain in convergence is at least optimal. For first order hyperbolic equations, technical assumptions of finite difference schemes are given in \cite{Gustafsson1975} under which the determinant condition is satisfied, thus one order is gained in convergence. A detailed analysis in \cite{Nissen2012,Nissen2013} for a class of discretizations for the Schr\"{o}dinger equation proves the gain is two orders, which is equal to the optimal gain. Results for the wave equation are presented in \cite{Wang2016}, where it is shown that the gain in convergence is not unified, but depends on boundary conditions and numerical boundary treatments. For both the Schr\"{o}dinger equation and the wave equation, the determinant condition is not satisfied in many of the cases considered, even though the schemes are stable. We remark that the theoretical convergence rate obtained from the normal mode analysis is in a generalized sense, analogue to the concept of stability in the generalized sense \cite[Chapter 12.3]{Gustafsson2013}. 

The accuracy analyses in the above mentioned references are limited to problems in one space dimension. However, the one dimensional analysis cannot always explain two dimensional numerical results. One example is two dimensional problems discretized in a multi--block setting, where numerical solutions in regions with different grid spacings are coupled using interpolation \cite{Kozdon2015,Mattsson2010,Nissen2015}. For cases like this, truncation errors located at grid points along the interface between two mesh blocks are of lower accuracy order compared to interior truncation errors. This is because of one--sided difference stencils on each side of the interface. At a few grid points on the edge of the interface, the accuracy of the truncation error is often one or two additional orders lower, caused by one--sided interpolations. We refer to such points as \textit{corner points}, shown in Figure \ref{fig:grid}(A).
The number of corner points depends on the particular discretization but is independent of grid spacing. Because the dominating truncation errors are localized in both spatial dimensions, it is a situation that does not occur in one dimensional problems. Numerical results in the literature indicate that such two dimensional cases may lead to higher convergence rates than what is predicted by the corresponding one dimensional analysis, see \cite{Kramer2009} for the advection equation, \cite{Nissen2012} for the Schr\"{o}dinger equation and \cite{Wang2016a} for the wave equation. The numerical results also indicate that depending on the partial differential equation and the numerical interface treatment, the gain in convergence rates may be different. Accuracy analysis for two dimensional problems is needed in order to fully understand these results.

\begin{figure}
	\subfloat[]{\includegraphics[width=0.33\textwidth]{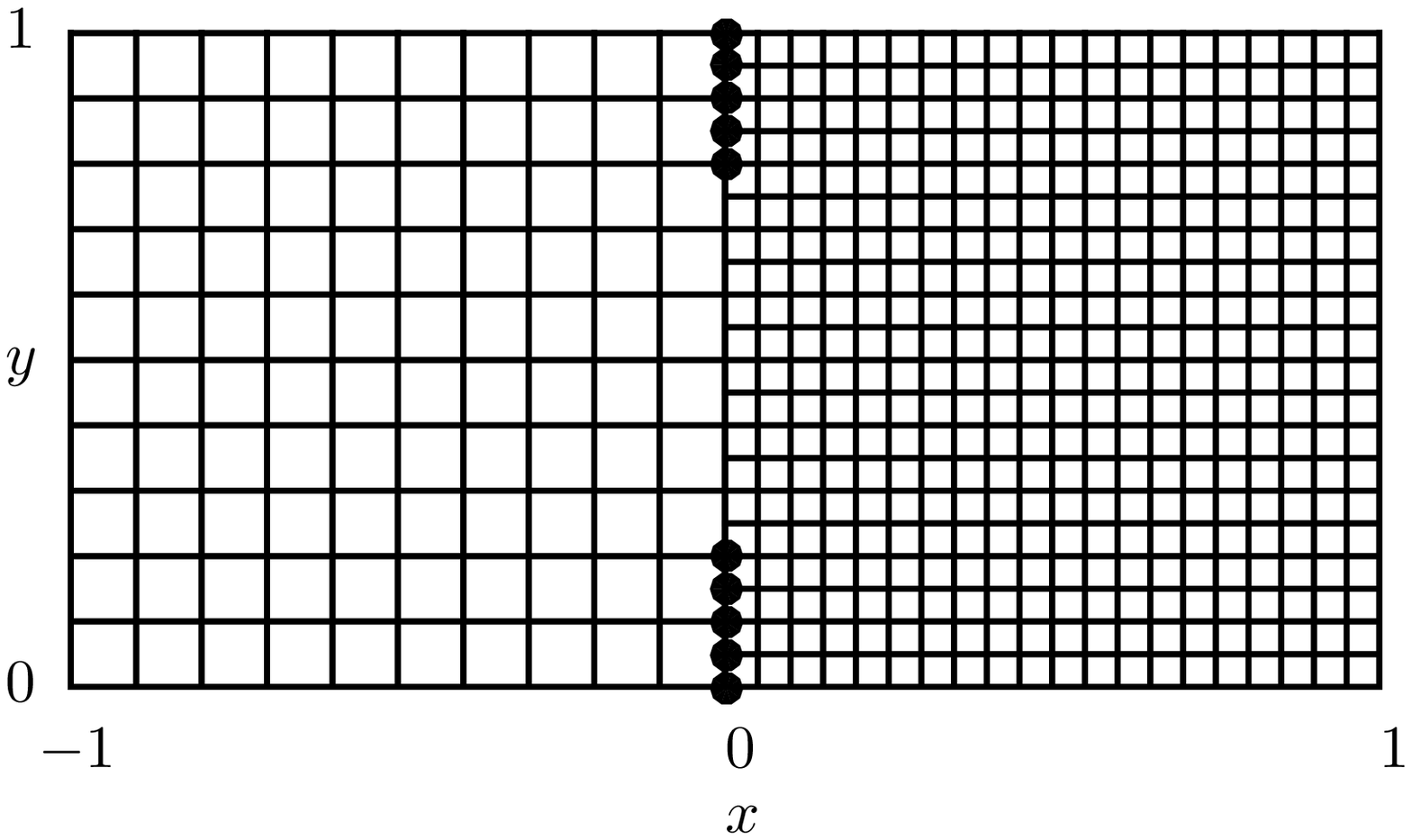}}
	\subfloat[]{\includegraphics[width=0.33\textwidth]{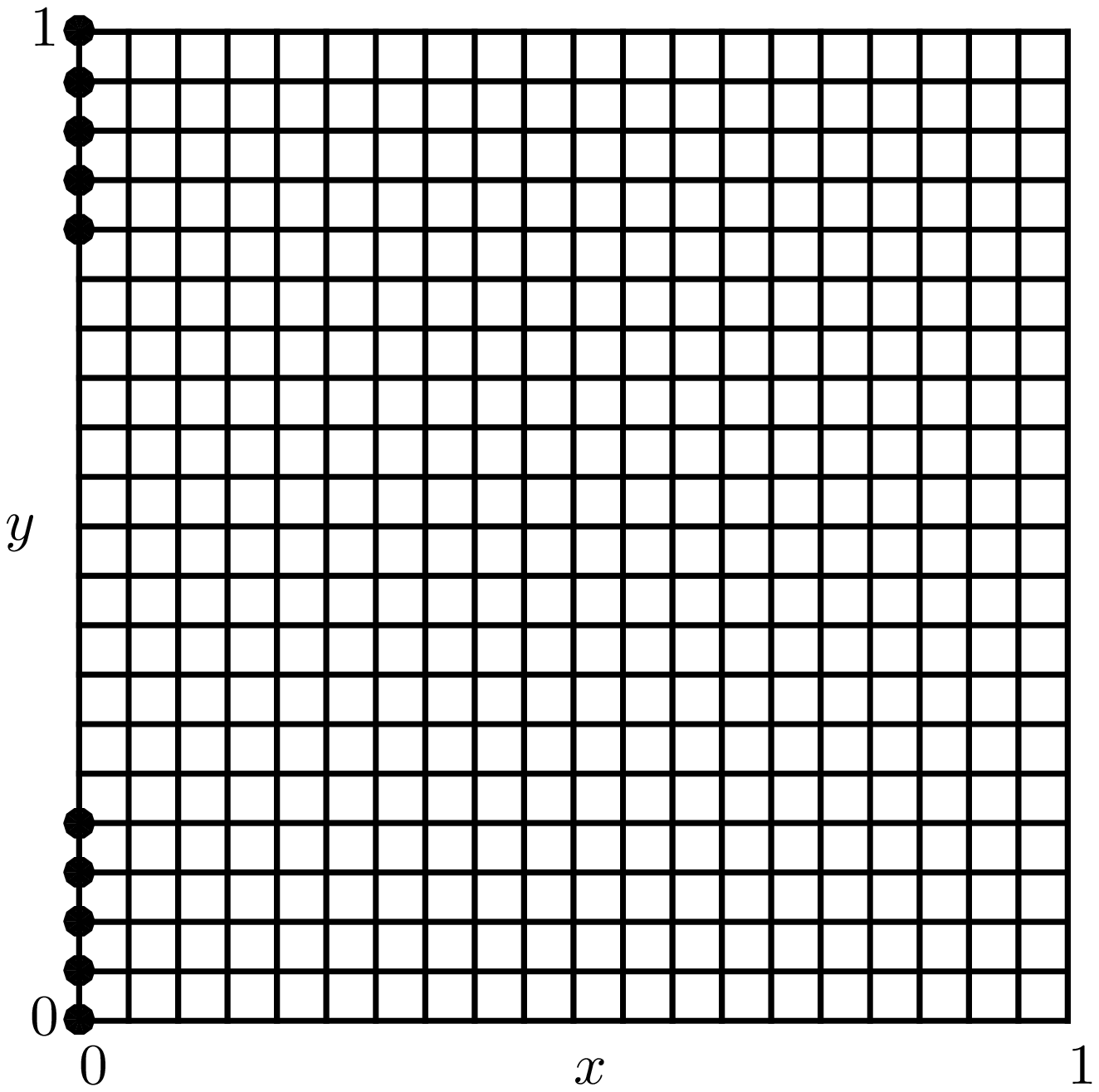}}
	\subfloat[]{\includegraphics[width=0.33\textwidth]{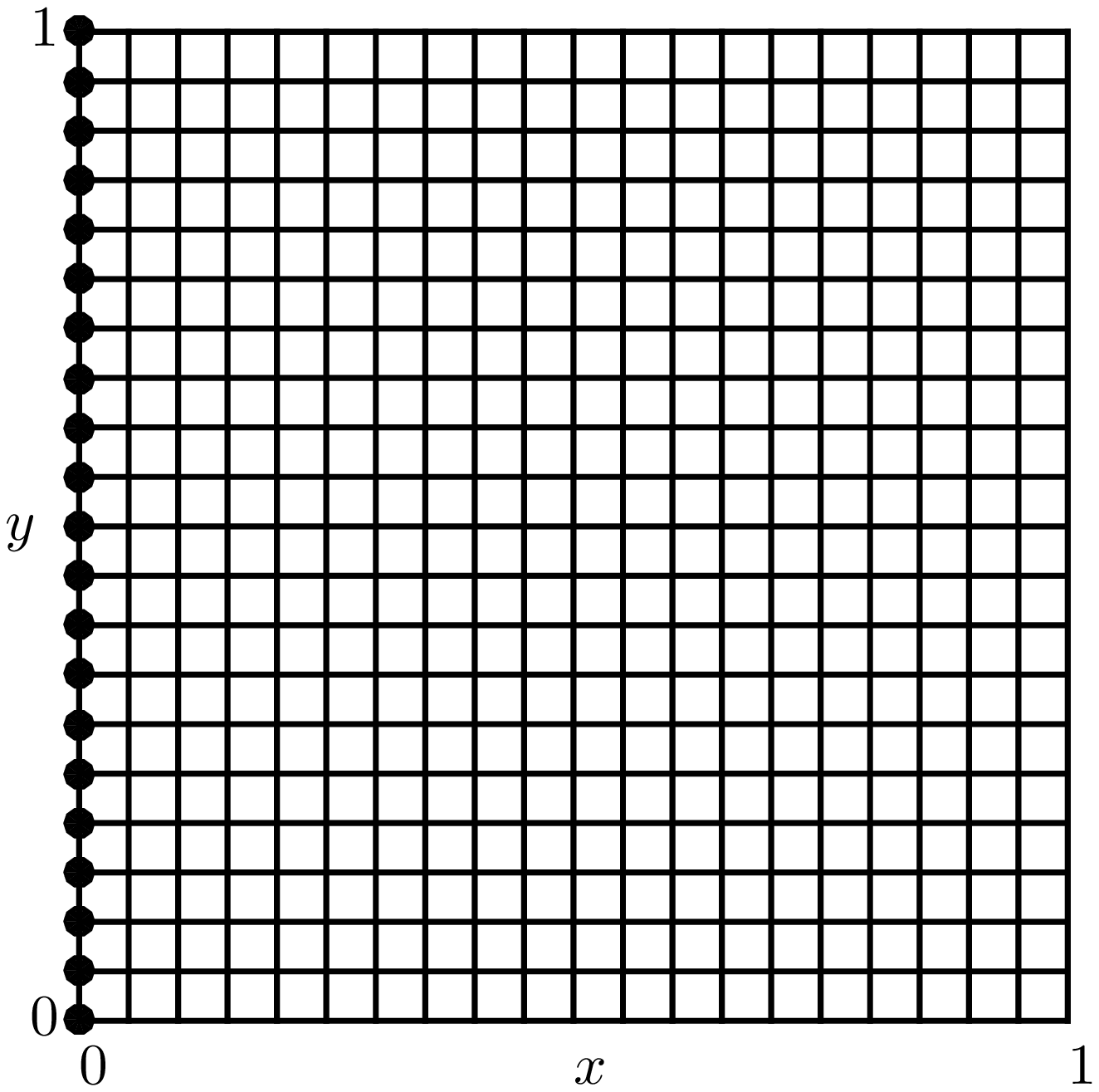}}
	\caption{(A) A multi--block grid with the dominating truncation error on the \emph{corner points} marked by filled circles. (B) A simplified model of the multi--block grid. (C) A single--block grid with the dominating truncation error along an entire boundary.}
	\label{fig:grid}
\end{figure}

In this paper, we present a general accuracy analysis framework for semi--discrete partial differential equations in two space dimensions, to better understand the effect localized truncation errors have on the overall spatial discretization error. As a model problem, we consider the semi--discretization of the wave equation by finite difference methods. High order methods solve  wave propagation problems more efficiently than low order methods on smooth domains \cite{Hagstrom2012,Kreiss1972}. However, it is a challenging task to construct stable and high order accurate methods for wave equations in the presence of boundaries and interfaces. One way to do this is to combine the summation--by--parts (SBP) finite difference method with the simultaneous approximation term (SAT) to impose boundary conditions. The SBP--SAT finite difference method has been successfully used to solve many types of differential equations numerically, and is our choice for the spatial discretization of the wave equation. However, the technique we develop to analyze accuracy is not limited to this class of methods. 

Since the equation is linear, we analyze the interior truncation error and boundary truncation error  separately. The interior truncation error can be analyzed straightforwardly by the energy method.  For the lower order boundary truncation error, we use again the superposition principle and analyze separately the effect of truncation errors along different parts of boundaries. Therefore, in the analysis we will focus on the truncation error along one boundary $x=0$ and consider two cases shown in Figure \ref{fig:grid}(B) and \ref{fig:grid}(C). In Figure \ref{fig:grid}(B), the dominating truncation error is located at only a few grid points on a boundary, i.e. localized in two space dimensions. This is a simplified model of the case shown in Figure \ref{fig:grid}(A). In addition, we also consider the case shown in Figure \ref{fig:grid}(C), which corresponds to when the dominating truncation error is located along an entire boundary. To further simplify the analysis, we will consider semi--infinite domains, where the boundaries at $x=1$ in Figure \ref{fig:grid}(B) and Figure \ref{fig:grid}(C) are moved to infinity. This separation of boundaries is justified by the same arguments which justify the separation in stability analysis \cite{Gustafsson2013}. 

The accuracy analysis relies on a transformation of the two dimensional problem to many one dimensional problems of the same type.  Each one dimensional problem is then analyzed by the normal mode analysis, a technique involving Laplace transformation in time. We show how for a one dimensional problem the error is dictated by the behaviour of a boundary system in the origin of the $s$--plane, where $s$ is the Laplace dual variable of time. This was discussed previously in \cite{Nissen2012,Nissen2013,Wang2016}, and in this paper we make the arguments more precise. We also show how to use the one dimensional results to obtain results for both two dimensional cases.

The outline of the paper is as follows. In Section 2 we present the two dimensional accuracy analysis, as well as the accuracy analysis for the relevant one dimensional problems. Properties for the semi--discrete system discretized with SBP operators and weak treatment of boundary conditions using the SAT method are described in Section \ref{sec-SBP}. Sections \ref{sec_NE} and \ref{sec_C} contain numerical experiments and conclusions, respectively.

\section{Accuracy analysis for the two dimensional wave equation}
We consider the two dimensional wave equation 
\begin{equation}\label{eqn_2d}
U_{tt}=U_{xx}+U_{yy}+F,
\end{equation}
on a domain 
\begin{equation*}
0\leq x<\infty,\quad 0\leq y\leq 1,\quad 0\leq t\leq t_f,
\end{equation*}
with suitable initial and boundary conditions so that \eqref{eqn_2d} is wellposed. We in particular consider the Dirichlet and Neumann boundary conditions. A comprehensive  discussion on the wellposedness of second order hyperbolic equations is found in \cite{Kreiss2012}.  Because the focus in this paper is the accuracy analysis of numerical methods solving equation \eqref{eqn_2d}, we assume that the solution is sufficiently smooth and in L$_2$ at any time.  This  gives rise to compatibility conditions between the initial and boundary data at the space--time corner. One way to guarantee that the compatibility conditions are satisfied is to assume that both spatial derivatives of the initial data on the boundary, and temporal derivatives of the boundary data at $t=0$, vanish up to sufficiently high order. We take this approach in the present work.

The domain is discretized by an equidistant grid with a grid spacing $h$ in both spatial directions
\begin{equation}\label{grid_1D}
\begin{split}
& x_i=(i-1)h,\ i=1,\cdots, \\
& y_j=(j-1)h,\ j=1,\cdots,N_y, 
\end{split}
\end{equation}
with $h=1/(N_y-1)$. As we will measure errors in both one dimensional and two dimensional spaces, we distinguish between three different discrete norms. Let $w$ be a two dimensional grid function on the grid \eqref{grid_1D}. We define the discrete norm as
\begin{equation*}
\|w\|_{2D}^2=h^2\sum_{i=1}^{N_y}\sum_{j=1}^{\infty} |w_{i,j}|^2.
\end{equation*}
We will also need norms of restrictions of grid functions to lines with constant $x$ or $y$ coordinates. They are denoted by 
\begin{equation*}
\|w_{i,:}\|_{1D,x}^2=h\sum_{j=1}^{\infty} |w_{i,j}|^2, \quad \|w_{:,j}\|_{1D,y}^2=h\sum_{i=1}^{{N_y}} |w_{i,j}|^2.
\end{equation*}

The spatial derivatives are approximated by finite difference operators yielding the semi--discretization 
\begin{equation}\label{semi_2d}
u_{tt}=\frac{1}{h^2}(Q_x\otimes I_y)u+\frac{1}{h^2}(I_x\otimes Q_y)u+F_h.
\end{equation}
In equation \eqref{semi_2d}, $u$ is a grid function approximating the true solution $U$ on the grid and is arranged column--wise, i.e. the first $N_y$ components of $u$ are the numerical solutions at  the grid points on the boundary $x=0$. The finite difference operators $Q_x/h^2$ and $Q_y/h^2$ approximate second derivatives in space,  including an implementation of the boundary conditions, and $I_x$, $I_y$ are identity operators. The Kronecker product $\otimes$ is used to extend the operators from one space dimension to two space dimensions. The grid function $F_h$ is the projection of the forcing function $F(x,y,t)$ on the grid. Inhomogeneous boundary data would also appear in the right--hand side of \eqref{semi_2d}, but is excluded here for a sake of simplified notation as it has no effect on the accuracy analysis. 

To consider accuracy, we need the semi--discretization \eqref{semi_2d} to be stable. One way to prove stability is to show by the energy method that the numerical solution in some appropriate norm is bounded by the data, and such a scheme is called energy stable. Another way is to prove stability in a generalized sense by the Laplace transform technique in the normal mode analysis framework, see details in \cite[Chapter 12]{Gustafsson2013}.  In particular, we consider operators $Q_y$ that satisfies the following property.

\begin{assumption}\label{antar_Qy}
There exists a symmetric positive definite operator $P$ such that with $Q_y$ in \eqref{semi_2d} the product $PQ_y$ is symmetric negative semi--definite, and the spectral norms of $P^{1/2}$ and $P^{-1/2}$ are uniformly bounded. 
\end{assumption}

The spectral norm, denoted by $\|\cdot\|$, is induced from the standard Euclidean vector norm. For any value $N_y$, operators in the $y$--direction can be represented by matrices. Assumption \ref{antar_Qy} on the operator $Q_y$ ensures that an energy estimate can be obtained when $Q_y$ is used as the spatial discretization operator in the corresponding one dimensional problem.  This assumption can be satisfied by many kinds of discretizations for which a standard energy estimate can be obtained. One example is a finite difference operator satisfying a summation--by--parts (SBP) property with a weak imposition of boundary conditions and properly chosen penalty parameters \cite{Appelo2007,Mattsson2009}. In this case, $H=hP$ is the operator associated with the SBP norm with the grid spacing $h$, and the condition number of $P^{1/2}$ is independent of $h$.  The SBP finite difference method is discussed in more detail in Section \ref{sec-SBP}.  

The following lemma describes an important property of the discretization operator $Q_y$ that will be needed for the two dimensional accuracy analysis.
\begin{lemma}\label{lemma_eig}
Consider the eigenvalue problem
\begin{equation}\label{eigen1}
\frac{1}{h^2}Q_y\varphi=-\lambda\varphi,
\end{equation}
with $Q_y$ in \eqref{semi_2d}. Under Assumption \ref{antar_Qy}, $Q_y$ is diagonalizable $Q_y/h^2=-\Phi\Lambda\Phi^{-1}$ by $\Phi=[\varphi_1,\varphi_2,\cdots,\varphi_{N_y}]$, where $\Lambda$ is a diagonal matrix with real and non--negative diagonal entries. In addition, $\|\Phi\|=\|P^{-1/2}\|$ and $\|\Phi^{-1}\|=\|P^{1/2}\|$ are uniformly bounded with respect to $N_y$.
\end{lemma}
The proof of Lemma \ref{lemma_eig} can be found in Appendix \ref{ProofLemmaEig}. As will be seen later, with such an operator $Q_y$ the error equation for the two dimensional problem can be transformed by a diagonalization technique to a number of one dimensional problems of the same type.

Let $U_h$ be the true solution $U$ projected on the grid, and the point--wise error be $\zeta(t)=U_h(t)-u(t)$. As discussed in the introduction, we will only analyze the effect of the truncation error, denoted $\mathcal{O}(h^p)$, caused by the one--sided stencil at the boundary $x=0$. The error equation is 
\begin{equation}\label{err_2D}
\zeta_{tt}=\frac{1}{h^2}(Q_x\otimes I_y) \zeta+\frac{1}{h^2}(I_x\otimes Q_y) \zeta+h^p T,
\end{equation}
where $h^p T$ is the boundary truncation error and $T$ is independent of $h$. Only the components of $T$ corresponding to the grid points located at $x=ih,\ i=0,\cdots k$ are nonzeros, where $k$ is a small constant independent of $h$. We therefore write 
\begin{equation}\label{Truncation_T}
T=[T_0;T_1;\cdots,T_k;\bold{0};\bold{0};\bold{0};\cdots],
\end{equation}
where $h^pT_i,\ i=0,\cdots k$, is the truncation error at the grid points located at $x=ih$. Here, $T$ is a two dimensional grid function, $T_i$ is a one dimensional grid function in the $y$--direction, and $\bold{0}$ is a zero vector of length $N_y$. 

\subsection{Diagonalization of the error equation} 
To begin with, we consider the case with the dominating truncation error along the entire boundary shown in Figure 1(C), for which  $\|T_i\|_{1D,y}^2=\mathcal{O}(1)$. In Section \ref{sec_localized_error}, we consider the localized case in Figure 1(B) with $\|T_i\|_{1D,y}^2=\mathcal{O}(h)$.

The next step of the normal mode analysis is to perform a Laplace transform in time of \eqref{err_2D}
\begin{equation}\label{err_L_2d}
s^2 \hat \zeta=\frac{1}{h^2}(Q_x\otimes I_y)\hat \zeta+\frac{1}{h^2}(I_x\otimes Q_y)\hat \zeta+h^p\hat T,
\end{equation}
where $s$ is the time dual in the Laplace space. We diagonalize $Q_y$ as in Assumption \ref{antar_Qy} and Lemma \ref{lemma_eig}, and rewrite \eqref{err_L_2d} as
\begin{equation*}
s^2 \hat \zeta=\frac{1}{h^2}(Q_x\otimes I_y)\hat \zeta-(I_x\otimes \Phi\Lambda\Phi^{-1})\hat \zeta+h^p\hat T,
\end{equation*}
where $\Lambda$ is diagonal with diagonal entries $\lambda^{(r)}\geq 0$ in ascending order for $r=1,2,\cdots,N_y$. 
Multiplying the above equation by $(I_x\otimes \Phi^{-1})$ from the left, we obtain
\begin{equation}\label{eqn_phie}
s^2(I_x\otimes\Phi^{-1})\hat \zeta=\frac{1}{h^2}(Q_x\otimes\Phi^{-1})\hat \zeta-(I_x\otimes\Lambda\Phi^{-1})\hat\zeta+h^p(I_x\otimes\Phi^{-1})\hat T.
\end{equation}
With the notation $\hat\epsilon=(I_x\otimes\Phi^{-1})\hat \zeta$, \eqref{eqn_phie} becomes
\begin{equation}\label{eqn_eps}
s^2\hat\epsilon=\frac{1}{h^2}(Q_x\otimes I_y)\hat\epsilon-(I_x\otimes\Lambda)\hat\epsilon+h^p(I_x\otimes\Phi^{-1})\hat T,
\end{equation}
with the operator in the $y$--direction diagonalized. This is the spectrally decomposed form, which consists of $N_y$ scalar difference equations
\begin{equation}\label{eqn_sca}
\underbrace{(\tilde s^2+h^2\lambda^{(r)})}_{(\tilde s_+^{(r)})^2}\hat\epsilon^{(r)}=Q_x\hat\epsilon^{(r)}+h^{p+2}\hat\tau^{(r)},
\end{equation}
where $r=1,2,\cdots,N_y$ and $\tilde s = sh$. For every $r$, we have
\begin{equation*}
\hat\tau^{(r)}=[\hat\tau^{(r)}_0,\hat\tau^{(r)}_1,\cdots,\hat\tau^{(r)}_k,0,0,0,\cdots]^T.
\end{equation*}

Note the close relation between $\hat\tau^{(r)}$  and  $\hat\tau_i=\Phi^{-1}\hat T_i$: the $i^{th}$ entry $\hat\tau^{(r)}_i$ is the same as the $r^{th}$ entry of $\hat\tau_i$. In addition, $\hat\epsilon$ in \eqref{eqn_eps} is related to $\hat\epsilon^{(r)}$ in \eqref{eqn_sca} by
\begin{equation}\label{sum_in_r}
\|\hat\epsilon\|_{2D}^2=h\sum_{r=1}^{N_y} \|\hat\epsilon^{(r)}\|_{1D,x}^2.
\end{equation}

With the notation $(\tilde s^{(r)}_+)^2=\tilde s^2+h^2\lambda^{(r)}$, we have transformed the two dimensional error equation to $N_y$ one dimensional error equations in the Laplace space. This transformation can be understood as a variant of Fourier transform. A general result on the estimate of the error $\zeta$ for the two dimensional problem is stated in the following theorem.

\begin{theorem}\label{MAIN}
If for all $s$ with Re$(s)=\eta>0$, $\hat\epsilon^{(r)}$ in \eqref{eqn_sca} is bounded as
\begin{equation}\label{estimate_r}
\|\hat\epsilon^{(r)}\|_{1D,x}^2\leq \frac{Kh^{2g}}{\eta^m} \sum_{i=0}^{k} \left( |\hat\tau^{(r)}_i|^2+\left|\widehat{\frac{\partial^b\tau^{(r)}_i}{\partial t^b}}\right|^2 \right),
\end{equation}
with $\eta,g,m$ and $b$ independent of $h$, then
\begin{equation}\label{MAIN_estimate}
\int_0^{t_f} \|\zeta\|^2_{2D} dt\leq\frac{K e^{2\eta t_f}h^{2g}}{\eta^{m}}\int_0^{t_f} \chi^2(\Phi) \sum_{i=0}^{k} \left(\|T_i\|_{1D,y}^2+\left\|\frac{\partial^b T_i}{\partial t^b}\right\|_{1D,y}^2 \right)dt.
\end{equation}
Here,  $\chi(\Phi)$ is the condition number of $\Phi$.
\end{theorem}

We note that in practice, to validate \eqref{estimate_r} we do not need to check for every $r$. It suffices to check the corresponding one dimensional problem, that is \eqref{estimate_r} without any shift when  $h^2\lambda^{(r)}=0$. To see this, we need the following lemma describing an important property of $\tilde s^{(r)}_+$.
\begin{lemma}\label{s+}
Let $\tilde s_+=\sqrt{\tilde s^2+\gamma}$ for some real $\gamma\geq 0$. If Re$(\tilde s)\geq\delta\geq 0$, then Re$(\tilde s_+)\geq\delta$.
\end{lemma}
The proof of Lemma \ref{s+} is given in Appendix \ref{2D_analysis}. This lemma implies that each error equation  \eqref{eqn_sca} can be seen as the corresponding one dimensional error equation perturbed  in the direction to the right of the complex plane. Therefore, the estimate of \eqref{eqn_sca} for $r=1,2,\cdots,N_y$ is no worse than the estimate for the corresponding one dimensional problem.

We will show in Section \ref{sec_background} how an estimate of the type \eqref{estimate_r} can be derived by normal mode analysis for a one dimensional problem of the type \eqref{eqn_sca}. In Theorem \ref{MAIN} and the rest of the paper, we use the capital letter $K$ in the estimates to denote some constant independent of the grid spacing $h$, where the precise value of $K$ could be different from one estimate to another. Below we give the proof of Theorem \ref{MAIN}.

\begin{proof}[Proof of Theorem \ref{MAIN}]
A sum of \eqref{estimate_r} in $r$ leads to 
\begin{equation*}
\|\hat\epsilon\|^2_{2D}\leq\frac{Kh^{2g}}{\eta^m} \sum_{i=0}^k \left(\|\tau_i\|_{1D,y}^2+\left\|\widehat{\frac{\partial^b\tau_i}{\partial t^b}}\right\|_{1D,y}^2 \right),
\end{equation*}
and Parseval's relation gives
\begin{equation}\label{estimate_physical_2d}
\int_0^{t_f} e^{-2\eta t}\|\epsilon\|^2_{2D} dt\leq\frac{Kh^{2g}}{\eta^m}\int_0^{t_f} e^{-2\eta t}\sum_{i=0}^k \left(\|\tau_i\|_{1D,y}^2+\left\|\frac{\partial^b\tau_i}{\partial t^b}\right\|_{1D,y}^2 \right)dt,
\end{equation}
where $\eta$ is a constant independent of $h$, $\epsilon=(I_x\otimes\Phi^{-1})\zeta$ and $\tau_i=\Phi^{-1}T_i$. It is obvious by the property of the induced norm that 
\begin{equation}\label{estimate_rhs}
\int_0^{t_f} e^{-2\eta t}\left(\|\tau_i\|_{1D,y}^2+\left\|\frac{\partial^b\tau_i}{\partial t^b}\right\|_{1D,y}^2\right)dt\leq \int_0^{t_f} e^{-2\eta t}\|\Phi^{-1}\|^2 \left(\|T_i\|_{1D,y}^2+\left\|\frac{\partial^b T_i}{\partial t^b}\right\|_{1D,y}^2\right)dt.
\end{equation}
In addition, we also have 
\begin{equation}\label{estimate_lhs}
\begin{split}
\|\Phi\|^2\int_0^{t_f} e^{-2\eta t}\|\epsilon\|^2_{2D} dt &= \int_0^{t_f} e^{-2\eta t}\|\Phi\|^2 \|(I_x\otimes\Phi^{-1})\zeta\|^2_{2D} dt \\
& \geq \int_0^{t_f} e^{-2\eta t} \|\zeta\|^2_{2D} dt.
\end{split}
\end{equation}
Finally, we obtain the following estimate by  \eqref{estimate_physical_2d}, \eqref{estimate_rhs} and \eqref{estimate_lhs},
\begin{equation}\label{estimate_physical_2d_f}
\int_0^{t_f}e^{-2\eta t} \|\zeta\|^2_{2D} dt\leq\frac{Kh^{2g}}{\eta^m}\int_0^{t_f} e^{-2\eta t}\chi^2(\Phi) \sum_{i=0}^k \left(\|T_i\|_{1D,y}^2+\left\|\frac{\partial^b T_i}{\partial t^b}\right\|_{1D,y}^2\right)dt,
\end{equation}
where $\chi(\Phi)=\|\Phi\| \|\Phi^{-1}\|$ is the condition number of $\Phi$, and by Lemma \ref{lemma_eig} it is uniformly bounded. The estimate \eqref{MAIN_estimate} is obtained by multiplying \eqref{estimate_physical_2d_f} with $e^{2\eta t_f}$ on both sides, by using $e^{-2\eta t}<1$ and $e^{2\eta (t_f-t)}>1$ for $0<t<t_f$.
\end{proof}

Note that by Lemma \ref{lemma_eig}, $\chi(\Phi)$ is uniformly bounded. Theorem \ref{MAIN} indicates that the convergence rate of a two dimensional problem is at least equal to that of the corresponding one dimensional problem. A special case for a two dimensional problem is when the dominating truncation error is localized  also in the $y$--direction, i.e. the number of grid points with the dominating truncation error is independent of $h$, yielding $\| T_i\|_{1D,y}^2=\mathcal{O}(h), i=0,\cdots,k$, in \eqref{MAIN_estimate}. Straightforwardly, this leads to an additional half order gain. A more detailed analysis in Section \ref{sec_localized_error} shows that the gain for a two dimensional problem with a localized truncation error could be a full order higher compared with the gain for a corresponding one dimensional problem.

\subsection{Background: accuracy analysis for one dimensional problems}\label{sec_background}
We would like to explore the relation between the one dimensional error equation \eqref{eqn_sca} and the corresponding estimate \eqref{estimate_r}. Therefore, accuracy analysis for the corresponding one dimensional problem is needed, and is presented in this section. Part of the analysis can also be found in \cite{Wang2016}, where it is argued that  the error estimate in Laplace space is determined by its behavior at the origin. In the paper, we give a proof of that.

We consider the one dimensional wave equation 
\begin{equation}\label{eqn_1d}
U_{tt}=U_{xx},\ 0\leq x<\infty,\ t\geq 0,
\end{equation} 
with appropriate initial and boundary conditions so that the problem is well--posed and the true solution is in L$_2$ and smooth. Similar to the two dimensional case, \eqref{eqn_1d} is discretized by a finite difference operator and the boundary conditions are imposed so that the semi--discretization is energy stable. The semi--discretization is 
\begin{equation}\label{1d_semi}
u_{tt}=\frac{Q}{h^2}u,
\end{equation}
 and the error equation in the Laplace space is
\begin{equation}\label{err_L_1d}
\tilde s^2\hat \zeta=Q\hat \zeta+h^{p+2}\hat T,
\end{equation}
where $s$ is the time dual in Laplace space and $\tilde s=sh$. Here $u$ and $\hat \zeta$ are vectors containing the numerical solution and point--wise error in the Laplace space, respectively. The operator $Q/h^2$ approximates the second derivative in space including an implementation of the boundary condition. The vector $\hat T$ has only a few non--zero components corresponding to the large truncation error close to the boundary $x=0$. Since the true solution is smooth, the non--zero components of $\hat T$  are to the leading order on the form
\begin{equation}\label{Tform}
\hat T_i=a_i\frac{\partial^{p+2}}{\partial x^{p+2}}\hat U(0,s),\ i=0,\cdots,k,
\end{equation}
where $a_i$ can be obtained by the Taylor expansion. The constants $a_i$ and $k$ are determined by the precise form of $Q$, and are independent of $h$.
\begin{remark}
Equation \eqref{err_L_1d} is in the same form as \eqref{eqn_sca}, with $\tilde s^2$ in the left--hand side of \eqref{err_L_1d} and a perturbed $(\tilde s_+^{(r)})^2$ in the left--hand side of \eqref{eqn_sca}. We keep this in mind in the following analysis. 
\end{remark}

The difference equation \eqref{err_L_1d} can be solved by first considering the components corresponding to the grid points away from the boundary where the forcing $h^{p+2}\hat T$ is zero and $Q/h^2$ is the standard central finite difference stencil. This step gives the characteristic equation with roots being functions of $\tilde s$. Since the problem under consideration is a half--line problem, in the estimate we only need to include the admissible roots $\kappa(\tilde s)$, for which $|\kappa(\tilde s)|<1$ for all Re$(\tilde s)>0$. With a $2l^{th}$ order central finite difference stencil, there are $l$ admissible roots $\kappa_1,\cdots,\kappa_l$. 

The next step in solving  \eqref{err_L_1d} is to consider its first few rows, corresponding to the discretization close to the boundary, where $h^{p+2}\hat T$ does not vanish. The corresponding equations can be written in a matrix--vector multiplication form
\begin{equation}\label{BS}
C(\tilde s)\Sigma=h^{p+2}\hat T_C,
\end{equation}
which is referred to as the boundary system. The matrix $C(\tilde s)$ is called the determinant matrix and depends on the boundary stencil, and is often of small size. For example, for the second order discretization in the SBP--SAT framework, $C(\tilde s)$ is a 3--by--3 matrix with the Dirichlet boundary condition, and a scalar with the Neumann boundary condition. For the precise form of these boundary systems, see equation (25) and (37) in \cite{Wang2016}, respectively. For higher order discretizations, the dimension of $C(\tilde s)$ is in general $k+1$. The first $k+1$ components of $\hat T_C$ is taken from $T_i,i=0,\cdots,k$ in \eqref{Tform}, with zeros appending to the end if needed. The unknown vector $\Sigma$ is
\begin{equation}\label{Sigma}
\Sigma=[\hat \zeta_1,\hat \zeta_2,\cdots, \hat \zeta_d, \sigma_1,\sigma_2,\cdots,\sigma_l]^T,
\end{equation}
where $d$ depends on the numerical boundary treatment. 

We can express the general solution of the difference equation \eqref{err_L_1d} by using  \eqref{Sigma}  as
\begin{equation}\label{sol_diff_1d}
\hat \zeta=[\hat \zeta_1,\hat \zeta_2,\cdots, \hat \zeta_d,\sum_{j=1}^l \sigma_j,\sum_{j=1}^l \sigma_j\kappa_j,\sum_{j=1}^l \sigma_j\kappa_j^2,\cdots]^T.
\end{equation}
Then the L$_2$ norm of $\hat \zeta$ is
\begin{equation}\label{L2_e_hat}
\begin{split}
\| \hat \zeta  \|^2_{1D,x} &=  h \sum_{i=1}^d |\hat \zeta_i|^2  +  h \sum_{j=1}^l |\sigma_j|^2 \sum_{k=0}^{\infty} |\kappa_j|^{2k} \\
&= h \sum_{i=1}^d |\hat \zeta_i |^2 + h  \sum_{j=1}^l  |\sigma_j|^2 \frac{1}{1- |\kappa_j|^{2}}.  
\end{split}
\end{equation}

There are two steps to derive an estimate of \eqref{L2_e_hat}. One step is to estimate $1/(1-|\kappa_j|^2)$, and is derived in Section \ref{sec-kappa}. In addition, we also need to estimate the other terms in \eqref{L2_e_hat}, i.e. the components of $\Sigma$:
\begin{equation}\label{2l_sol}
|\hat \zeta_i|,\ i=1,2,\cdots,d \text{ and } |\sigma_j|,\ j=1,2,\cdots,l.
\end{equation}
As will be seen later, when $\tilde s=0$ the characteristic equation always has an admissible root equal to 1, denoted by $\kappa_1=1$. Therefore, we consider Re$(\tilde s)=\eta h=\mathcal{O}(h)$ in the limit as $h$ approaches zero. The solution to \eqref{BS} can be written as
\begin{equation}\label{SigmaE}
\Sigma=h^{p+2}C^{-1}(\tilde s)\hat T_C, \text{ Re}(\tilde s)>0,
\end{equation} 
because of the following lemma.
\begin{lemma}\label{lemma_stability}
If the discretization \eqref{1d_semi} is stable, then $C(\tilde s)$ in the boundary system \eqref{BS} is non--singular for all Re$(\tilde s)>0$. 
\end{lemma}
The proof of Lemma \ref{lemma_stability} is very similar to the proof of Lemma 12.1.1 in \cite[pp.~378]{Gustafsson2013}, but for completeness we include it in Appendix \ref{app_stability}.

In the following section, we derive estimates for each component of $\Sigma$, denoted by $\Sigma_i$, where $i=1,\cdots,l+d$. For simplified notation, we use the maximum norm of vectors
\begin{equation}\label{Sigma_max}
|\Sigma_i| \leq \|\Sigma\|_{\max} \leq h^{p+2} \|C^{-1}(\tilde s) \hat T_C\|_{\max}  \leq h^{p+2} \|C^{-1}(\tilde s)\|_{\max} \|\hat T_C\|_{\max},
\end{equation}
where $\|\Sigma\|_{\max} = \max_i |\Sigma_i|$. We emphasize again that the dimension of the vector $\Sigma$ in \eqref{SigmaE} is finite and is independent of the grid spacing. In certain cases, we do not use the last inequality in \eqref{Sigma_max} because some components of $\Sigma$ may vanish and computing $C^{-1}(\tilde s) \hat T_C$ directly gives a sharper estimate.

\subsubsection{Estimates of $|\Sigma_i|$}\label{sec-Sigma}
When $\tilde s$ is away from the imaginary axis, $C(\tilde s)$ is non--singular by the energy stability. In many common stable semi--discretizations of the second order wave equation, the determinant matrix $C(\tilde s)$ is singular for some $\tilde s$ on the imaginary axis. We will show below that singularities away from the origin typically have no influence on the order of accuracy, and can be handled separately. A singularity at $\tilde s=0$ on the other hand can introduce a loss of accuracy compared with the optimal gain. The  following lemma makes the arguments precise.

\begin{lemma}\label{antar_C2}
Consider the boundary system \eqref{BS} at $\tilde s=i\tilde\xi+\eta h$ where $\eta>0$ is a constant independent of $h$. Let $\delta$ be a small but $h$--independent constant. 
\begin{itemize}
\item When $\tilde s$ is in the vicinity of the origin, i.e. $|\tilde s|\leq \delta$, there exists a non--negative integer $w$ such that 
\begin{equation}\label{sigma_w}
|\Sigma_i|\leq h^{p+2-w}\frac{K}{\eta^w}\|\hat T_C\|_{\max}.
\end{equation}
If $C(0)$ is nonsingular, then $w=0$. Otherwise, $w$ can be determined by the precise form of the boundary system \eqref{BS}. 
\item When $\tilde s$ is away from the origin, i.e. $|\tilde s|>\delta$, there exists a non--negative integer $\alpha$ such that 
\begin{equation}\label{sigma_alpha}
|\Sigma_i|\leq h^{p+2}\frac{K}{(\eta\delta)^\alpha}\left\|\widehat{\frac{\partial^\alpha T_C}{\partial t^\alpha}}\right\|_{\max}.
\end{equation}
The value of $\alpha$ can be determined by the precise form of the boundary system \eqref{BS}. 
\end{itemize}
\end{lemma}
The proof of Lemma \ref{antar_C2} is found in Appendix \ref{proof_C2}.

If the boundary system \eqref{BS} is singular at the origin, then we need to use \eqref{sigma_w} to estimate $|\Sigma_i|$ leading to a $w^{th}$ order loss. If the singularity occurs on the imaginary axis away from the origin, \eqref{sigma_alpha} introduces no accuracy order loss, but temporal derivatives of the truncation error appear in the estimate of $|\Sigma_i|$ .  As an example, with Neumann boundary conditions  we have $w=1$ for a second order method and $w=0$ for a fourth order method, see \cite{Wang2016}.

\subsubsection{Estimates of $1/(1-|\kappa_j|^2)$}\label{sec-kappa}
The characteristic equation and its solution only depend on the finite difference stencil in the interior of the computational domain. To discretize the wave equation \eqref{eqn_1d}, it is common to use the standard central finite difference stencil as the interior stencil. In this case, we have the following lemma for the roots.

\begin{lemma}\label{antar_kappa}
Consider the characteristic equation corresponding to the difference equation $\tilde s^2\zeta=\tilde Q\zeta$, where $\tilde Q$ is a standard central difference operator approximating the second derivative. Let $\tilde s=i\tilde\xi +\eta h$ with a constant $\eta>0$ independent of $h$. The roots of the characteristic equation have the following properties. 
\begin{itemize}
\item When $\tilde s$ is in the vicinity of the origin, there is one admissible root $\kappa_1(\tilde s)$ to the characteristic equation satisfying 
\begin{equation}\label{d}
\frac{1}{1-|\kappa_1(\tilde s)|^2}\leq\frac{K}{\eta h}.
\end{equation}
The constant $K$ depends on the finite difference stencil, but not on $h$ or $\eta$. For the other admissible roots, $\frac{1}{1-|\kappa(\tilde s)|^2}$ is bounded independently of $h$. 
\item When $\tilde s$ is away from the origin, there can be admissible roots satisfying 
\begin{equation}\label{kappa_beta}
\frac{1}{1-|\kappa_j(\tilde s)|^2}\leq\frac{K}{(\eta h)^\beta}.
\end{equation}
for some $\beta>0$. Such roots do not introduce any accuracy order loss when estimating \eqref{L2_e_hat}. 
\end{itemize}
\end{lemma}
The proof of Lemma \ref{antar_kappa} can be found in Appendix \ref{proof_kappa}. When $\tilde s$ is in a vicinity of the origin, the bound on $1/(1-|\kappa_1(\tilde s)|^2)$ in \eqref{d} leads to a $h^{-1}$ factor in \eqref{L2_e_hat}. When $\tilde s$ is away from the origin, $1/(1-|\kappa_j(\tilde s)|^2)$ does not introduce any $h$--dependent factor in \eqref{L2_e_hat}. Together with Lemma \ref{antar_C2}, we therefore conclude that the order of accuracy is determined by the error in the vicinity of the origin in Laplace space. 

\subsubsection{Final estimates in the Laplace space}
The final error estimate is determined by the combined effect of the estimates \eqref{sigma_w} and \eqref{d} from Lemma \ref{antar_C2} and  \ref{antar_kappa}, respectively. In the worst scenario, we have 
\begin{equation}\label{sigma_kappa}
|\sigma_1|\sim h^{p+2-w} \text{ and } 1/(1-|\kappa_1|^2)\sim h^{-1},
\end{equation}
which leads to 
\begin{equation}\label{estimate_Laplace}
\|\hat \zeta\|_{2D}^2\leq \frac{Kh^{2p+4-2w}}{\eta^{1+2w+b}}\left(\|\hat T_C\|_{\max}^2+\left\|\widehat{\frac{\partial^b T_C}{\partial t^b}}\right\|_{\max}^2\right),
\end{equation}
where the constant $K$ is independent of $h$. A more detailed analysis might reveal that $\sigma_1/(1-|\kappa_1|^2)$ satisfies a sharper bound than given in \eqref{sigma_kappa}. The value $b=2\alpha+\beta$ in \eqref{estimate_Laplace} is given by $\alpha$ in \eqref{sigma_alpha} and $\beta$ in \eqref{kappa_beta}. Comparing with the estimate \eqref{estimate_r} in Theorem \ref{MAIN} we have $2g=2p+4-2w$ and $m=1+2w+b$. 

\subsection{Accuracy analysis for two dimensional problems with corner truncation errors}\label{sec_localized_error}
We now consider a two dimensional problem where the large truncation error along a boundary is only localized at a few grid points at a corner, see Figure \ref{fig:grid}(B).

Without loss of generality we consider only non-zero elements in $T_0$, i.e. the truncation error has the form \eqref{Truncation_T} with $k=0$ and $T_0 = [\mathcal{O}(1), 0, \cdots]^T$. An immediate consequence is that $\|T_0\|_{1D,y}\sim\mathcal{O}(h^{1/2})$, which together with the uniform boundedness of $\|\Phi^{-1}\|$ in Lemma \ref{lemma_eig}, leads to 
\begin{equation}\label{tau_norm} 
\|\hat\tau_0\|_{1D,y}\leq \|\Phi^{-1}\|\|\hat T_0\|_{1D,y}\leq K h^{1/2}.
\end{equation}
 A direct application of Theorem \ref{MAIN} to this problem thus leads to an additional gain of a half order in convergence, compared with when the truncation error is present along the entire boundary. However, numerical experiments for certain problems show another half order better result. Below we demonstrate how to sharpen the analysis so that a full order gain in convergence is obtained compared with the case when the dominating truncation error is located along an entire boundary. 

One difference between the $r^{th}$ error equation \eqref{eqn_sca} and the standard one dimensional error equation \eqref{err_L_1d} is that the coefficient $\tilde s^2$ in \eqref{err_L_1d} is replaced by $(\tilde s^{(r)}_+)^2=\tilde s^2+h^2\lambda^{(r)}$ in \eqref{eqn_sca}. We can therefore view $h^2\lambda^{(r)}$ as a perturbation to the error equation, which enables us to use the one dimensional results presented in Section \ref{sec_background}. To obtain a sharp estimate for a two dimensional problem when the large truncation error is located only at a few grid points, we cannot rely on a uniform estimate in $r$ as in  \eqref{estimate_r}. Instead, the discrete eigenvalues $\lambda^{(r)}$ are divided into two sets according to indices $ r = 1,\cdots, r_{\delta}$ and $r = r_{\delta}+1, \cdots, N_y$, where $r_{\delta}$ is such that
\begin{align}
&\sqrt{h^2\lambda^{(r)}} \leq \delta,\quad r = 1,\cdots, r_{\delta}, \label{rdelta1} \\
&\sqrt{h^2\lambda^{(r)}} > \delta,\quad r = r_{\delta}+1, \cdots, N_y, \label{rdelta2}
\end{align}
for some constant $\delta > 0$ independent of $h$. Since we have a consistent spatial discretization the discrete eigenvalues in equation \eqref{eigen1}, $\lambda^{(r)}$, converge to the eigenvalues, $\lambda_c$, of the corresponding continuous problem
\begin{align}
\label{cont_eig}
\frac{\partial^2}{\partial y^2}\varphi_c =- \lambda_c \varphi_c,
\end{align}
where \eqref{cont_eig} is closed with the same boundary conditions as the ones approximated by the discrete operator in \eqref{eigen1}. Therefore, $h^2\lambda^{(r)}$ for a particular $r$ decreases when $h$ decreases, and thus $r_{\delta}$ grows with grid refinement. 

It is helpful to first consider a specific example. Let 
\[Q_y = \begin{bmatrix}
-1& 1&  &&\\
1& -2 &1 &  & \\
 & 1& -2& 1&  \\
&&\ddots &\ddots& \ddots
\end{bmatrix},\]
which corresponds to the standard second order accurate discretization of the Laplace operator in one dimension with Neumann boundary conditions. The operator $Q_y/h^2$ is diagonalized by the standard unitary cosine transform, for which we have the operator $\Phi$ defined by
\begin{align}
\Phi_{i,r} = \left\{ \begin{tabular}{l}
$\frac{1}{\sqrt{N_y}}$, $i = 1, \cdots, N_y, \ r=1$, \notag \\
$\sqrt{\frac{2}{N_y}} \cos(\frac{\pi (r-1) (i-1/2)}{N_y})$,  \ $i = 1, \cdots, N_y, \quad r = 2, \cdots, N_y$, \notag \\ 
\end{tabular} \right.
\end{align}
and the discrete eigenvalues are given by 
\begin{equation}\label{continuous_eig}
\lambda^{(r)} = \frac{4}{h^2} \sin^2 \left(\frac{\pi (r-1)}{2N_y} \right),\ r = 1, \cdots, N_y.
\end{equation}
Note that $\lambda^{(r)}$, converge to the eigenvalues, 
\begin{equation*}
\lambda_c^{(r)}=(r-1)^2\pi^2,\ r=1,2,\cdots,
\end{equation*}
of \eqref{cont_eig} with Neumann boundary conditions. Furthermore, from the cosine transform of $(1, 0, \cdots)^T$, we immediately see that
\begin{align}
\hat \tau_0 = \frac{\hat T_{0,0}}{\sqrt{N_y}} \left(1, \sqrt{2}\cos\left(\frac{\pi}{2 N_y}\right),  \sqrt{2}\cos\left(\frac{2 \pi}{2 N_y}\right),  \sqrt{2}\cos\left(\frac{3 \pi}{2 N_y}\right), \cdots \right)^T.
\end{align}
Here we note the agreement with the uncertainty principle \cite{Pinsky2009}, which states that the more concentrated a function is, the more spread out its Fourier transform must be. In other words, locality in $\hat T_0$ rules out locality in $\hat\tau_0$.   

We shall now consider a more general case, where $Q_y/h^2$ is a consistent spatial discretization satisfying Assumption \ref{antar_Qy}. An analytical formula of $\lambda^{(r)}$ may not exist, but the  eigenvalues of the discrete operator converge to the eigenvalues of the corresponding continuous problem \eqref{cont_eig} with the same boundary conditions. To ensure a corresponding locality principle in the more general case, we make the following assumption on the spatial discretization in the $y$--direction.

\begin{assumption}\label{antar_tau}
The diagonalized transformation in Lemma \ref{lemma_eig} satisfies the uncertainty principle in the sense that when $T_0 = [\mathcal{O}(1), 0, \cdots]^T$ then each component of $\hat \tau_0 = \Phi^{-1} T_0$ satisfies
\begin{equation}\label{tau_bound}
|\hat \tau_{0}^{(r)}| = \mathcal{O}(h^{1/2}).
\end{equation} 
\end{assumption}

By \eqref{tau_norm} we know that $\|\hat\tau_0\|_{1D,y}$ is proportional to $h^{1/2}$. Assumption \ref{antar_tau} implies a component--wise bound \eqref{tau_bound} of $\hat\tau_0^{(r)}$, and rules out the possibility of a local character of the truncation error also in the spectral representation. Assumption \ref{antar_tau} is a reasonable assumption because when the operator $\Phi^{-1}$ acts on $\hat T_0$, it is an analogue of performing a discrete Fourier transform. It is a discrete Fourier transform if the boundary condition in the $y$ direction is periodic. In the following analysis, we use the component--wise bound for $r\leq r_{\delta}$.

The total error for the two dimensional Laplace--transformed and spectrally decomposed problem is given by
\begin{equation}\label{2D_error_transformed}
\| \hat \epsilon\|_{2D}^2 = h \sum_{r=1}^{N_y} \| \hat \epsilon^{(r)}\|_{1D,x}^2= h \sum_{r=1}^{N_y} \left( h \sum_{n=1}^{d} |\hat \epsilon^{(r)}_n|^2 + h \sum_{j=1}^l\ \frac{|\sigma^{(r)}_j|^2}{1-|\kappa^{(r)}_j|^2}\right). 
\end{equation}
Note the similarity between \eqref{2D_error_transformed} and the expression for $\| \hat \zeta\|_{2D}^2$ in \eqref{L2_e_hat}. To obtain sharp accuracy results we divide \eqref{2D_error_transformed} into  
\begin{align}
\| \hat \epsilon^{(r \leq r_{\delta})}\|_{2D}^2 = h \sum_{r=1}^{r_{\delta}} \|\hat \epsilon^{(r)}\|_{1D,x}^2,
\end{align}
and
\begin{align}
\| \hat \epsilon^{(r > r_{\delta})}\|_{2D}^2 = h \sum_{r=r_{\delta}+1}^{N_y} \|\hat \epsilon^{(r)}\|_{1D,x}^2,
\label{large_r_sum}
\end{align}
and estimate them separately. For $r \leq r_{\delta}$, by the component--wise bound \eqref{tau_bound} an additional $1/2$ order is gained in \eqref{sigma_w}. We have
\begin{align}\label{sigma_half}
&|\hat \epsilon_n^{(r)}| \leq h^{p+5/2-w}\frac{K}{\eta^w},  \quad n= 1, \cdots, d,\\
&|\sigma_j^{(r)}| \leq h^{p+5/2-w}\frac{K}{\eta^w},  \quad  j = 1, \cdots, l, \notag 
\end{align}
where $w$ is related to the invertibility of $C(\tilde s)$ in Lemma \ref{antar_C2}. In 
\begin{equation*}
\| \hat \epsilon^{(r \leq r_{\delta})}\|_{2D}^2  
=  h \sum_{r=1}^{r_\delta} h \sum_{n=1}^{d} |\hat \epsilon^{(r)}_n|^2 + h \sum_{r=1}^{r_\delta} h \sum_{j=1}^l\ \frac{|\sigma^{(r)}_j|^2}{1-|\kappa^{(r)}_j|^2},
\end{equation*}
the first term can be bounded as 
\begin{equation}\label{term1}
 h \sum_{r=1}^{r_\delta} h \sum_{n=1}^{d} |\hat \epsilon^{(r)}_n|^2\leq h^{2p+7-2w}\sum_{r=1}^{r_\delta} \frac{K}{\eta^{2w}},
\end{equation}
while the second term
\begin{equation}\label{term2}
h \sum_{r=1}^{r_\delta} h \sum_{j=1}^l\ \frac{|\sigma^{(r)}_j|^2}{1-|\kappa^{(r)}_j|^2} \leq h^{2p+6-2w}\frac{K}{\eta^{2w}} \sum_{j=1}^l \sum_{r=1}^{r_\delta} \frac{h}{1-|\kappa_j^{(r)}|^2}.
\end{equation}

By Lemma \ref{antar_kappa}, only one admissible root leads to an $h^{-1}$ factor in the error estimate.  We therefore only consider $1/(1-|\kappa_1^{(r)}|^2)$ in \eqref{term2}. Together with $(\tilde s_+^{(r)})^2=\tilde s^2+h^2\lambda^{(r)}$, we obtain
\begin{align}
\frac{1}{1-|\kappa_1^{(r)}(\tilde s_+^{(r)})|^2} \leq \frac{K}{(\sqrt{\eta^2 + \lambda^{(r)}}) h}.
\end{align}

With a common boundary condition like Dirichlet, Neumann or periodic, the continuous eigenvalues $\sqrt{\lambda_c}$ in \eqref{cont_eig} are uniformly distributed. For small $r$ corresponding to the well--resolved components, the discrete eigenvalues converge to the continuous ones. As a consequence, $\sqrt{\lambda^{(r)}} \approx K r$ for $r < r_{\delta}$. This can also be seen from the formula  \eqref{continuous_eig}. For $\eta = 0$ and $\sqrt{\lambda^{(r)}} \approx K r$, we have 
\begin{align}
\sum_{r=1}^{r_{\delta}} h \frac{1}{1-|\kappa_1^{(r)}|^2} \approx h K \left( \frac{1}{h} + \frac{1}{2h} + \cdots + \frac{1}{r_{\delta} h} \right) \leq K \left( \log( h^{-1}) + \gamma + \mathcal{O}(h) \right),
\end{align}
where $r_{\delta} h = \delta$ is a constant independent of $h$ and $\gamma \approx 0.58$ is the Euler--Mascheroni constant. This leads to
\begin{equation}\label{estimate_small_r}
\| \hat \epsilon^{(r \leq r_{\delta})}\|_{2D}^2 \leq  h^{2p+6-2w}  \log( h^{-1}) \frac{K}{\eta^{2w}}.
\end{equation}

Let us now consider $r > r_{\delta}$. In this case, the shift in $\tilde s_+^{(r)}$ is $h^2\lambda^{(r)}>\delta^2\sim\mathcal{O}(1)$. Therefore, all $\kappa_j^{(r)}$ have absolute values bounded away from 1, which leads to the estimate $1/(1- |\kappa_j^{(r)}|^2) \leq K$ for a constant $K$ independent of $h$. We then have the estimate
\begin{align*}
\| \hat \epsilon^{(r > r_{\delta})}\|_{2D}^2 &= h \sum_{r=r_{\delta}+1}^{N_y}  \|\hat \epsilon^{(r)}\|_{1D,x}^2\\
&\leq K h \sum_{r=r_{\delta}+1}^{N_y} (h\sum_{i=1}^d |\hat\epsilon^{(r)}|^2+h |\sigma_1^{(r)}|^2) \\
&\leq \frac{K}{\eta^{2w}}  h \sum_{r=r_{\delta}+1}^{N_y}  h^{2p+5-2w}\|C^{-1}(\tilde s_+^{(r)})\|_{\max}^2 |\hat\tau_0^{(r)} |^2 \\
&\leq \frac{K}{\eta^{2w}}  h^{2p+5-2w} \max_{r} \|C^{-1}(\tilde s_+^{(r)})\|_{\max}^2 \|\hat\tau_0\|_{1D,y}^2\\
&\leq \frac{K}{\eta^{2w}}  h^{2p+6-2w} \max_{r} \|C^{-1}(\tilde s_+^{(r)})\|_{\max}^2.
\end{align*}
Because of the shift in $\tilde s_+^{(r)}$, the determinant matrix $C(\tilde s_+^{(r)})$ is nonsingular by stability. Hence, $\max_{r} \|C^{-1}(\tilde s_+^{(r)})\|^2_{\max}$ is a constant of order $\mathcal{O}(1)$. 

Together with \eqref{estimate_small_r}, we have the combined estimate
\begin{equation*}
\| \hat \epsilon\|_{2D}\leq  K \log (h^{-1}) h^{p+3-w},
\end{equation*}
and consequently
\begin{equation}\label{estimate_Laplace_2D}
\| \hat \zeta\|_{2D}\leq K \log (h^{-1}) \|\Phi\| h^{p+3-w}\leq K \log (h^{-1}) h^{p+3-w}.
\end{equation}
In the above estimate, we use the uniform boundedness of $\|\Phi\|$ in Lemma \ref{lemma_eig}, and $K$ depends on the true solution but not on $h$. The factor $\log(h^{-1})$ grows with mesh refinement, but much slower than $h^{-1}$. Asymptotically, only the exponent of $h$ in  \eqref{estimate_Laplace_2D} is important. Comparing the two dimensional case with a corner truncation error with the one dimensional problem, we can expect a full order gain in convergence from \eqref{estimate_Laplace} to \eqref{estimate_Laplace_2D}.

\section{Summation--by--parts finite difference methods}\label{sec-SBP}
High order finite difference methods solve wave propagation problems more efficiently than low order methods on smooth domains \cite{Hagstrom2012,Kreiss1972}. Though standard central finite difference stencils can be used in the interior of the domain, it is challenging to derive stable and accurate schemes close to boundaries and interfaces. An approach that has been successfully used to overcome this difficulty is finite difference operators satisfying a summation--by--parts (SBP) property \cite{Kreiss1974,Kreiss1977} in combination with the simultaneous--approximation--term (SAT) technique \cite{Carpenter1994} to impose boundary and interface conditions. A main advantage of the SBP--SAT method is its unified procedure to construct both provably stable and highly accurate schemes for linear time--dependent problems with boundaries and interfaces. Recent developments also include its use on non--uniform grids \cite{Fernandez2014b}, and its relation to a form of discontinuous Galerkin method \cite{Gassner2013}. Introductions of the SBP--SAT finite difference method can be found in \cite{Fernandez2014,Svard2014}.

\subsection{Semi--discretizations by the SBP operators}
To approximate a second derivative, we use SBP operators constructed in \cite{Mattsson2004}. They are central finite difference stencils in the interior of the computational domain and special one--sided stencils at a few grid points near boundaries. The operator $D\approx\partial^2/\partial x^2$  is called an SBP operator if it can be decomposed as
\begin{equation*}
D=H^{-1}(-M+BS),
\end{equation*}
where $H$ is a diagonal positive definite operator associated with an L$_2$--equivalent norm, the operator $M$ is symmetric positive semi--definite and $B$ takes the form $\text{diag}(-1,0,0,\cdots,1)$. The first and last row of $S$ approximate the first derivative at the boundary grid points.

The truncation error of the SBP operators in \cite{Mattsson2004} is $\mathcal{O}(h^{2p})$ in the interior and $\mathcal{O}(h^p)$ near boundaries, $p=1,2,3,4$. Though the accuracy is sacrificed near boundaries in order to satisfy the SBP property, the operators are often termed as $2p^{th}$ order accurate. 

In the following, we present the SBP--SAT schemes for the wave equation with Dirichlet and Neumann boundary conditions. For comparison, we also state the result of the accuracy analysis for one dimensional problems, which was derived in \cite{Wang2016}.

\subsubsection{Dirichlet boundary conditions}
The semi--discretization of the one dimensional wave equation \eqref{eqn_1d} with a Dirichlet boundary condition is 
\begin{equation}\label{semi_D}
u_{tt}=Du-H^{-1}S^T(E_0u-\bar g)-\frac{\iota}{h} H^{-1}(E_0u-\bar g)+F_g,
\end{equation}
where $E_0$ picks up $u$ at the boundary with entries zero almost everywhere except $E_0(1,1)=1$. The vector $\bar g=[g(t),0,\cdots]^T$ contains the boundary data and $F_g$ is the restriction of the forcing $F(x,t)$ onto the grid. The constant $\iota$ is independent of $h$ but must satisfy $\iota\geq\iota_0$ for stability \cite{Appelo2007}, and its lower bound $\iota_0$ is presented in \cite{Mattsson2008,Mattsson2009,Wang2016}. A stability proof of \eqref{semi_D} by the energy method can be found in \cite{Mattsson2009}. We also note that if the boundary condition  of equation \eqref{eqn_2d} at $y=0,1$ is Dirichlet and the discretization \eqref{semi_D} is used also in the $y$ direction, then the operator $Q_y$ in \eqref{semi_2d} satisfies Assumption \ref{antar_Qy} with $P=h^{-1}H$. 

\subsubsection{Neumann boundary conditions}
For the wave equation \eqref{eqn_1d} with a Neumann boundary condition, it is natural to impose the boundary condition weakly. The corresponding SBP--SAT scheme is 
\begin{equation}\label{semi_N}
u_{tt}=Du+H^{-1}(E_0Su-\bar g)+F_g,
\end{equation}
and a stability proof can be found in \cite{Mattsson2009}. We note that when equation \eqref{eqn_2d} has a Neumann boundary condition at $y=0,1$ and the discretization \eqref{semi_N} is used also in the $y$ direction, the operator $Q_y$ in \eqref{semi_2d} satisfies Assumption \ref{antar_Qy} with $P=h^{-1}H$. 

\subsubsection{Accuracy analysis for the one dimensional wave equation}\label{Accuracy1D}
The accuracy analysis of both the discretizations \eqref{semi_D} and \eqref{semi_N} are performed in detail in \cite{Wang2016}, and we summarize the main result in Table \ref{tab:1}. 
\begin{table}
\centering
\caption{Result of accuracy analysis for the one dimensional problem. $2p:$ order of the interior truncation error.  $p:$ order of the boundary truncation error. $q=\min(2p,p+\text{Gain}):$ overall convergence rate. $\iota,\iota_0:$ energy stability is obtained if the penalty parameter $\iota$ satisfies $\iota\geq\iota_0$.}
\label{tab:1}       
\begin{tabular}{c c c c c c c}
\hline\noalign{\smallskip}
 & \multicolumn{2}{c}{ $2p=2$} & \multicolumn{2}{c}{ $2p=4$}   & \multicolumn{2}{c}{ $2p=6$}    \\
 & $p+$Gain & $q$ & $p+$Gain &$q$ & $p+$Gain &$q$\\
\noalign{\smallskip}\hline\noalign{\smallskip}
Dirichlet $(\iota=\iota_0)$ & 1+0.5 & 1.5 & 2+0.5 &2.5 & 3+0.5 &3.5 \\
Dirichlet $(\iota>\iota_0)$ & 1+2 &2 & 2+2 &4 & 3+2.5 &5.5 \\
Neumann & 1+1 &2 & 2+2 &4 & 3+2.5 &5.5 \\
\noalign{\smallskip}\hline
\end{tabular}
\end{table}

For the Dirichlet problem, the penalty parameter $\iota$ in \eqref{semi_D} has an influence on the convergence rate. If $\iota=\iota_0$, the determinant condition is not satisfied, and the gain in convergence is only a half order; if $\iota>\iota_0$, the determinant condition is satisfied and $w=0$ in the estimate \eqref{estimate_Laplace}, leading to the optimal two orders gain. In fact, super--convergence is obtained with the sixth order accurate scheme, where the gain in convergence is two and a half orders. The reason of super--convergence is found by a careful analysis of the boundary system \eqref{sigma_w}, which shows that $\sigma_1\sim h^{p+3}$ in \eqref{sigma_kappa}.

For the Neumann problem, the determinant condition is always violated. A careful analysis in \cite{Wang2016} shows that the gain in convergence is 1, 2, and 2.5 for the second, fourth and sixth order accurate schemes, respectively. For the fourth and sixth order accurate schemes, the gain in convergence is at least the optimal two orders. This is because the boundary system in the normal mode analysis has a special structure, which leads to $w=0$ in the estimate \eqref{estimate_Laplace}.  We refer to a detailed discussion in \cite{Wang2016}.   We also remark that if the structure of the boundary system is perturbed by modifying the truncation error, the gain in convergence is one order in all three cases. 

By Theorem \ref{MAIN}, for the SBP--SAT approximation of the Dirichlet and Neumann problem in two space dimensions, the convergence rates are the same as the corresponding one dimensional cases shown in Table \ref{tab:1}.

\section{Numerical experiments}\label{sec_NE}
In this section, we perform numerical experiments to verify our accuracy analysis, with a focus on two different kinds of truncation errors for the wave equation in two space dimensions.  The first case is when the large truncation error is located on an entire boundary, which typically occurs in a single--block domain or a multi--block domain with conforming grid interfaces. Such an experiment is presented in Section \ref{experiment_standard}. In the second case, the truncation error is only located at a few corner points in a two dimensional domain and the number of such grid points is independent of grid refinement, see Section \ref{experiment_nonstandard}. This experiment is constructed as a simplified analogue to a multi--block setting with non--conforming grid interfaces presented in \cite{Wang2016a}. The dominating truncation error is located on a few corner points,  see an illustration in Figure  \ref{fig:grid}(A). More precisely, the $2p^{th}$ order accurate schemes have a dominating truncation error $\mathcal{O}(h^{p-2})$ because of interpolation. 

\subsection{A standard two dimensional problem}\label{experiment_standard}
We consider the two dimensional wave equation \eqref{eqn_2d} in the computational domain $[0,1]^2$ discretized on a uniform grid. The SBP finite difference operators constructed in \cite{Mattsson2004} are used to approximate the spatial derivatives, and the SAT technique is used to weakly impose all the boundary conditions according to the discretizations \eqref{semi_D} and \eqref{semi_N}. The semi--discretizations \eqref{semi_D} and \eqref{semi_N} are generalized to two space dimensions using the Kronecker product. We will consider both Dirichlet and Neumann problems, meaning that all boundary conditions are either Dirichlet or Neumann, respectively.

For hyperbolic problems, explicit methods are often used to advance the equation in time. In the following numerical experiments, we employ the fourth order Runge--Kutta method as the time integrator. In \cite{Kreiss1993}, it is shown that under reasonable conditions the fully discrete scheme is stable if the semi--discretization is stable and Runge--Kutta methods are used to discretize in time. Since we investigate the convergence rate in space, we choose the time step $\Delta t=0.1h$ small enough so that the temporal error is negligible compared with the spatial error. This step size is well below the stability limit. The final time of the simulations is chosen to be $t=2$.

The numerical solution is compared with an analytic solution constructed using the method of manufactured solutions:
\begin{equation}\label{sol_analytical}
U(x,y,t)=\cos(10\pi x+1)\cos(10\pi y+2)\cos(10\pi\sqrt{2}t+3).
\end{equation}
The L$_2$ errors are computed as the norm of the difference between the exact solution projected onto the grid with grid spacing $h$, $U_{h}$, and the corresponding numerical solution, $u_h$, according to
\begin{equation*}
\|u_h-U_{h}\|_{\text{L}_2}=h\sqrt{(u_h-U_{h})^T(u_h-U_{h})},
\end{equation*}
and the convergence rate is computed by
\begin{equation*}
q=\log\left(\frac{\|u_h-U_{h}\|_{\text{L}_2}}{\|u_{2h}-U_{2h}\|_{\text{L}_2}}\right) \bigg/ \log\left(\frac{1}{2}\right).
\end{equation*}
In the numerical experiments, we observe that the computed L$_2$ convergence rates are in agreement with the rates in the generalized sense obtained in the analysis. 

In the SBP--SAT scheme, the penalty parameter for the Dirichlet problem, $\iota$, is chosen to be $20\%$ larger than its lower bound required by energy stability. It is important to choose $\iota$ larger than the threshold $\iota_0$ to obtain the desired rate of convergence, as the choice $\iota=\iota_0$ gives suboptimal convergence. This is demonstrated in \cite{Wang2016}. We note, however, that an extremely large $\iota$ leads to a very small time step. An increase of $20\%$ of the penalty parameter seems appropriate based on our experiments.
\begin{figure}
	\subfloat[]{\includegraphics[width=0.49\textwidth]{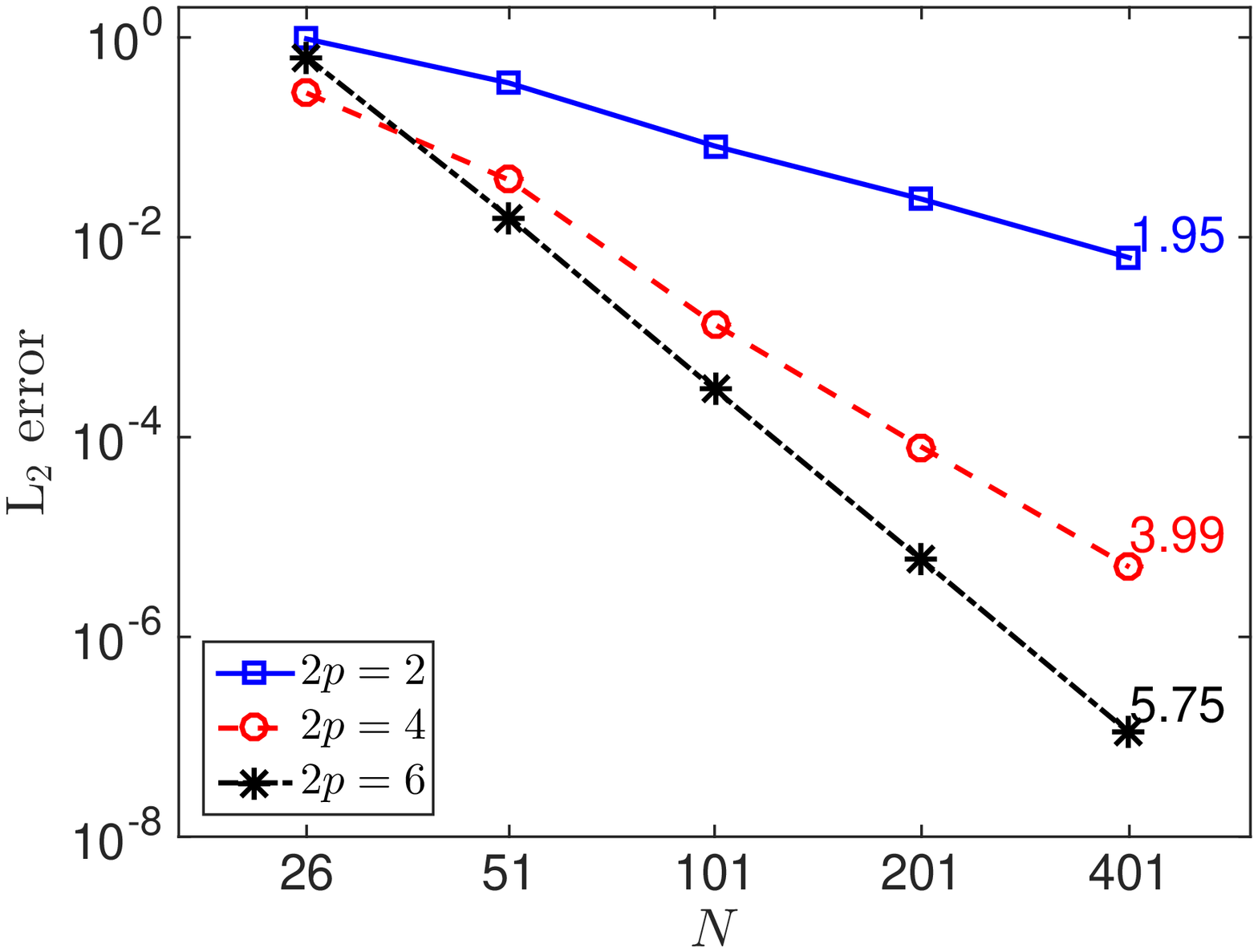}}
	\subfloat[]{\includegraphics[width=0.49\textwidth]{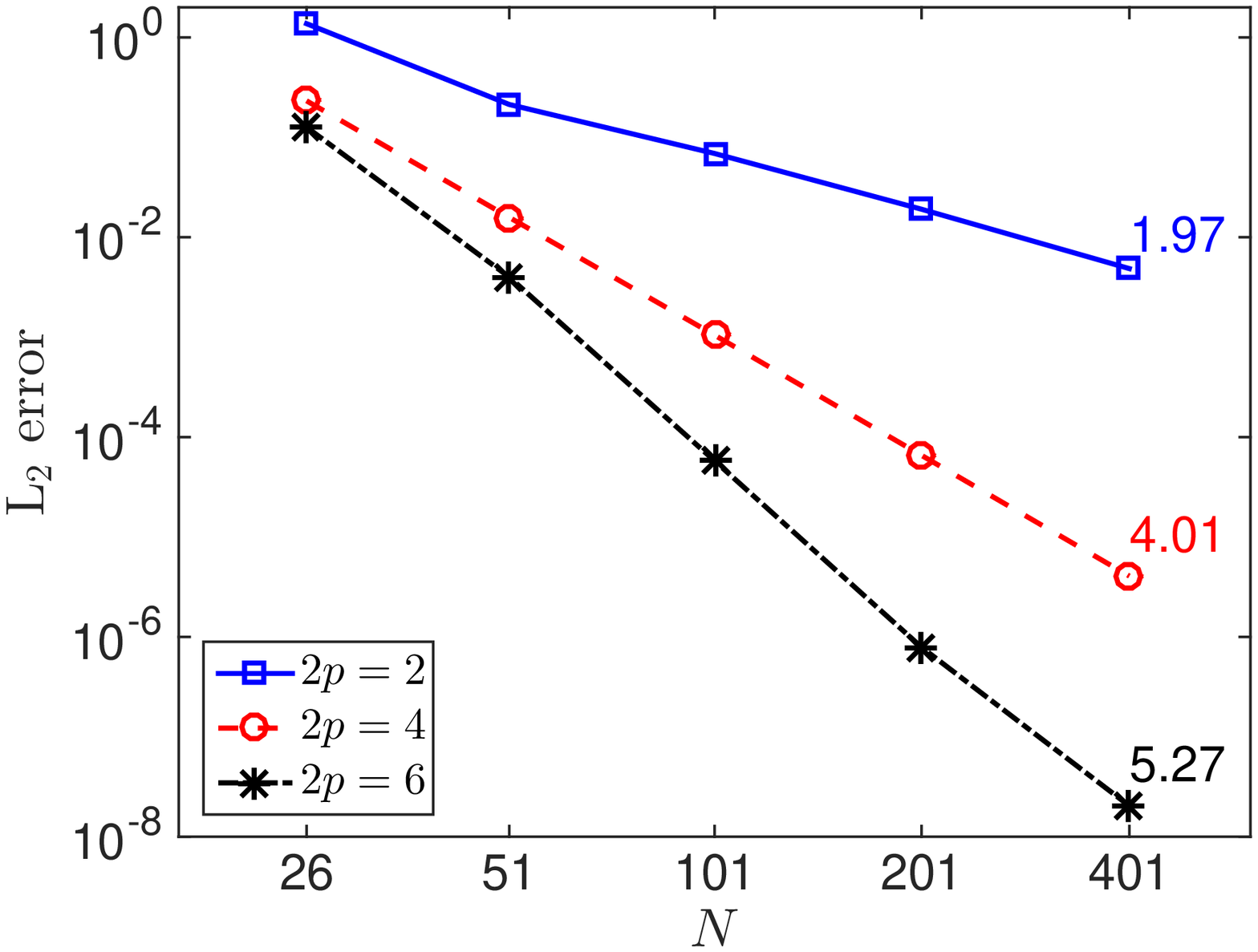}}
	\caption{Experiments for the two dimensional wave equation with (A) Dirichlet boundary conditions (B) Neumann boundary conditions.}
	\label{fig:2D}
\end{figure}

The L$_2$ errors versus the number of grid points in each spatial dimension, $N$, are plotted in Figure \ref{fig:2D}. The convergence rates computed from the last two mesh refinements for each set of experiments are shown at the end of the corresponding convergence curves in the error plots. According to Theorem \ref{MAIN}, Table \ref{tab:1} and the analysis in \cite{Wang2016}, for both the Dirichlet and Neumann problem, the SBP--SAT approximations have theoretical convergence rates 2, 4 and 5.5 for the second, fourth and sixth order accurate schemes, respectively. As seen in Figure \ref{fig:2D}, the convergence rates for the corresponding two dimensional problem are 2 and 4 for the second and fourth order accurate methods, which agree well with our analysis. The convergence rate for the sixth order accurate method is 5.75 for the Dirichlet problem and 5.27 for the Neumann problem. 

\subsection{Truncation error located on a few corner points }\label{experiment_nonstandard}
The case analyzed in Section \ref{sec_localized_error} is a simplified model of the multi--block finite difference discretization with non--conforming grid interfaces. Below we conduct  numerical experiments for this simplified model to verify that the analysis in  Section \ref{sec_localized_error} is sharp.

We use the  same setting as in Section \ref{experiment_standard}, but in this case we consider the manufactured, smooth solution
\begin{equation}\label{Exact_local}
U(x,y,t)=\cos(4x+1)\cos(4y+2)\cos(4\sqrt{2}t+3)
\end{equation}
to the two dimensional wave equation \eqref{eqn_2d}, and solve until $t=2$ using the methods of interior order $2p=2,4,6$.

To get a truncation error $\mathcal{O}(h^{p-2})$ at a few corner points, as shown in Figure 1(B), we modify the standard SBP--SAT scheme by using erroneous boundary data. For the Dirichlet problem, the true boundary data at the grid point $(0,y_i)$, denoted by $g_{i}(t)$, can be obtained from  \eqref{Exact_local} as $g_i(t)=U(0,y_i,t)$. However, in the penalty terms in the SAT method, we use $(1+\nu)g_i(t)$ on the first and last five grid points. In particular, for the $2p^{th}$ order method we choose $\nu=c_ph^{p},\ p=1,2,3$, where the factors $c_p$ are chosen to match the coefficients of the truncation error by one--sided difference stencils. Since the boundary data is multiplied by $H^{-1}S^T\sim\mathcal{O}(h^{-2})$, this choice results in a truncation error $\mathcal{O}(h^{p-2})$. We note that the stability property of the discretization does not change. We plot the L$_2$ error versus the number of grid points in each space dimension in Figure \ref{fig:corner}(A). The convergence rates 2.01, 3.07 and 4.21 imply a gain of three orders in all the three cases, and are in agreement with the error estimate \eqref{estimate_Laplace_2D}. 

\begin{figure}
	\subfloat[]{\includegraphics[width=0.49\textwidth]{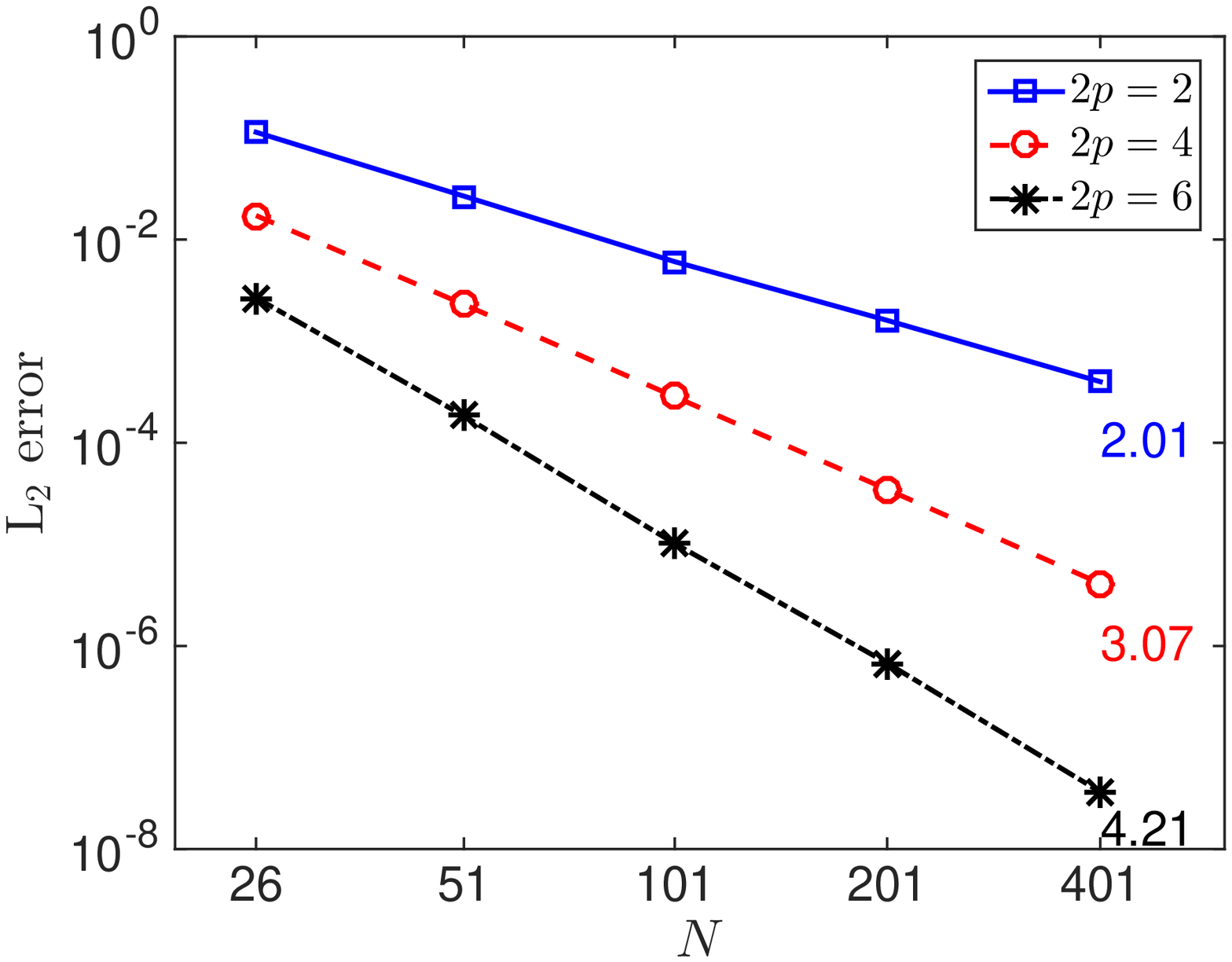}}
	\subfloat[]{\includegraphics[width=0.49\textwidth]{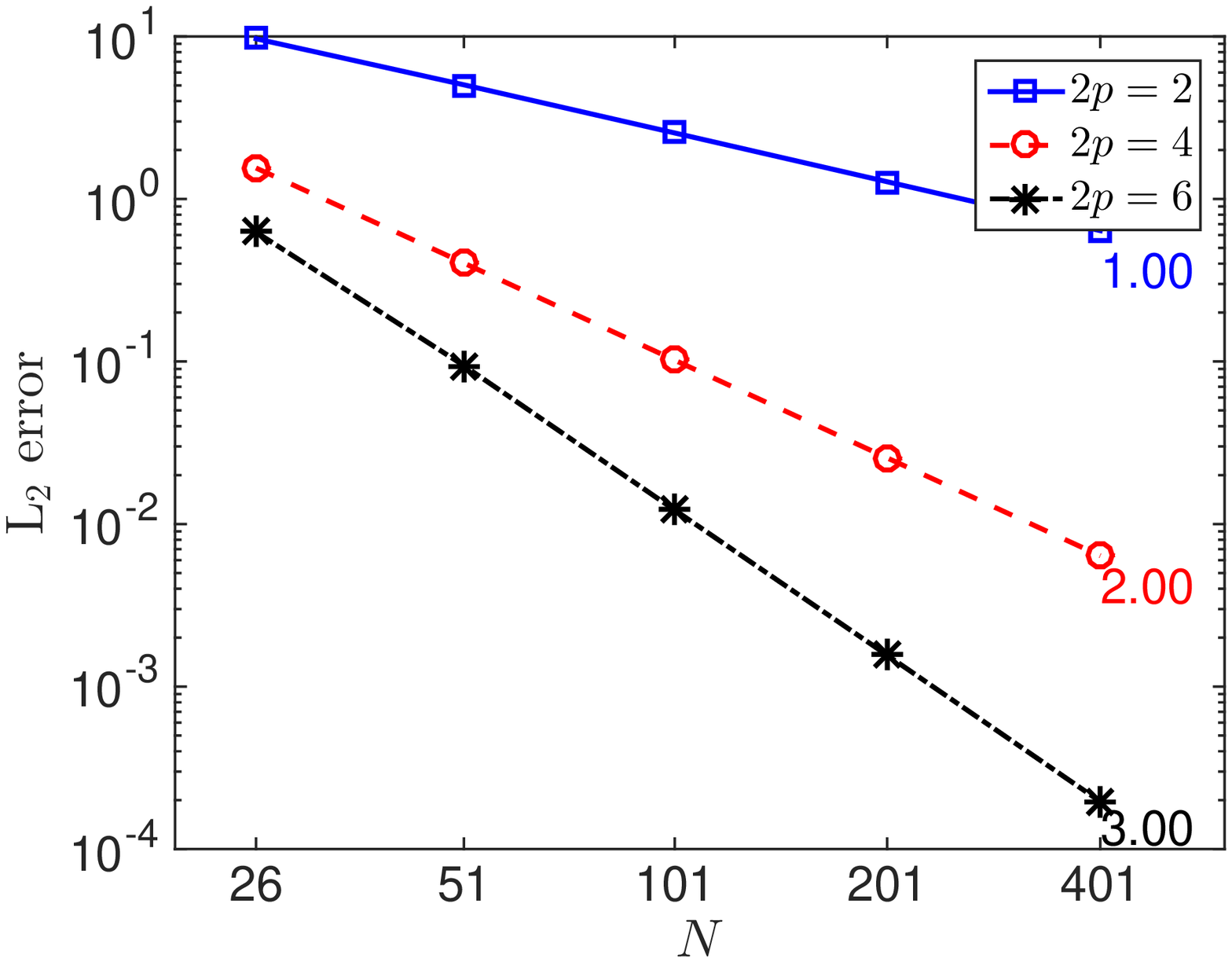}}
	\caption{ $L_2$ errors for (A) the Dirichlet  problem (B) the Neumann problem. The dominating  truncation error is $\mathcal{O}(h^{p-2})$ located at ten grid points as shown in Figure \ref{fig:grid}(B).}
	\label{fig:corner}
\end{figure}

For the Neumann problem we again use erroneous boundary data, this time by perturbing the true data by adding $\nu=c_ph^{p-1},\ p=1,2,3$. Here the data is multiplied by $H^{-1}\sim\mathcal{O}(h^{-1})$ in the penalty term, hence the truncation error at those ten grid points is as in the Dirichlet case $\mathcal{O}(h^{p-2})$. As is mentioned towards the end of Section \ref{Accuracy1D}, the perturbation changes the structure of the right--hand side of the boundary system.  The analysis for one dimensional problems presented in \cite{Wang2016} leads to $w=1$ in \eqref{estimate_Laplace_2D} for all the three cases, and a corresponding gain of two orders for the two dimensional problem. The L$_2$ errors  from the numerical computations are plotted in Figure \ref{fig:corner}(B). The convergence rates are 1.00, 2.00 and 3.00. In all the three cases, the gain is two orders, and agrees with the accuracy analysis leading to the estimate \eqref{estimate_Laplace_2D}. For neither the Dirichlet nor the Neuman problem do we observe any effect of the logarithmic term in \eqref{estimate_Laplace_2D}, but this is not surprising since asymptotically such an effect would be difficult to detect.


\section{Conclusion}
\label{sec_C}
In this paper, we extend the accuracy analysis of finite difference methods solving initial--boundary--value problems to two space dimensions. The two dimensional analysis is based on a diagonalization technique to decompose a two dimensional problem into one dimensional problems of the same type.  We then continue the analysis by utilizing the results from the one dimensional analysis. We have chosen the second order wave equation as the model problem, but the technique presented in this paper can be used to analyze other equations. 

The second contribution of this paper is the analysis of the effect of truncation errors localized  at a few grid points in a corner of a two dimensional domain. This kind of truncation errors often occur in multi--block finite difference discretizations with non--conforming grid interfaces. The analysis is performed for a simplified but analogous problem, a single block with large truncation errors at  a few grid points close to a corner. The numerical experiments for the simplified problem show that our accuracy analysis is sharp in the limit as the grid spacing approaches zero. Numerical experiments in a multi--block setting presented in \cite{Wang2016a} also agree well with this analysis.

In addition, we have presented a detailed framework of analyzing the convergence rate of one dimensional problems by the normal mode analysis, and have shown that the critical point is to derive sharp error estimates in the vicinity of $s=0$ in Laplace space. Singularities of the boundary system can also occur on the imaginary axis away from the origin, but they have no influence on the final convergence rate. 

\section*{Acknowledgements}
We would like to thank the help from Stefano Serra--Capizzano on Lemma \ref{lemma_eig}, and Bengt Fornberg on Lemma \ref{antar_kappa}, respectively. This work was performed when the second author was at the University of Bergen and was supported by VISTA (project 6357) in Norway. The first author and the third author are partially supported by the Swedish Research Council (project 106500511).

\appendix
\section{Proof of Lemma \ref{lemma_eig}}\label{ProofLemmaEig}
\begin{proof}
Let $\tilde Q_y=PQ_y$, then the eigenvalue problem \eqref{eigen1} can be written
\begin{equation}\label{eigen11}
\frac{1}{h^2}P^{-1}\tilde Q_y\varphi=-\lambda\varphi.
\end{equation}
Since $P$ is symmetric positive definite, $P^{1/2}$ and $P^{-1/2}$ are also symmetric positive definite. We rewrite \eqref{eigen11} as 
\begin{equation*}
\frac{1}{h^2}P^{-1/2}\tilde Q_yP^{-1/2}P^{1/2}\varphi=-\lambda P^{1/2}\varphi.
\end{equation*}
With the notation $\hat Q_y=P^{-1/2}\tilde Q_yP^{-1/2}$ and $\hat\varphi=P^{1/2}\varphi$, we obtain the following new eigenvalue problem 
\begin{equation}\label{eigen2}
\frac{1}{h^2}\hat Q_y\hat\varphi=-\lambda\hat\varphi.
\end{equation}
Because $\tilde Q_y$ is symmetric negative semi--definite, according to Sylvester's law of inertia, $\hat Q_y$ is also symmetric negative semi--definite. Therefore, the eigenvalues of $\hat Q_y$ are real and non--positive. The matrices $Q_y$ in \eqref{eigen1} and $\hat Q_y$ in \eqref{eigen2} are similar because $\hat Q_y=P^{1/2} Q_y P^{-1/2}$, so they have the same eigenvalues. This proves that $\lambda\geq 0$ in \eqref{eigen1}.

Let $\hat\Phi$ be the unitary operator $[\hat\varphi_1,\hat\varphi_2,\cdots,\hat\varphi_{N_y}]$, then its condition number in spectral norm is equal to one, i.e. $\chi(\hat\Phi)=1$. Because of $\hat\varphi=P^{1/2}\varphi$, we have $\hat\Phi=P^{1/2}\Phi$. As a consequence, equation \eqref{eigen1} can be written as
\begin{equation*}
\frac{1}{h^2}Q_y(P^{-1/2}\hat\Phi)=-(P^{-1/2}\hat\Phi)\Lambda.
\end{equation*}
where $\Lambda=\text{diag}(\lambda_1,\lambda_2,\cdots,\lambda_{N_y})$. Because both $P^{-1/2}$ and $\hat\Phi$ are invertible, $P^{-1/2}\hat\Phi$ is also invertible. Therefore, $Q_y$ is diagonalizable as
\begin{equation}\label{diagQy}
\frac{1}{h^2}Q_y=-(P^{-1/2}\hat\Phi)\Lambda(P^{-1/2}\hat\Phi)^{-1}.
\end{equation}
Since the spectral norm is unitarily invariant and $\Phi=P^{-1/2}\hat\Phi$, we have 
\begin{equation*}
\|\Phi\|=\|P^{-1/2}\| \text{ and } \|\Phi^{-1}\|=\|P^{1/2}\|.
\end{equation*}
By Assumption \ref{antar_Qy}, $\|\Phi\|$ and $\|\Phi^{-1}\|$ are uniformly bounded.
\end{proof}

\section{Proof of Lemma \ref{s+} }
\label{2D_analysis}
\begin{proof}
For the special case $\delta=0$, the lemma is equivalent to Lemma 4 in \cite{Wang2016}, and its proof is found there. In the following, we only consider the case when $\delta>0$. Let $\tilde s=a+bi$ and $\tilde s_+=c+di$, where $a,b,c,d$ are real numbers and $a\geq\delta> 0$.  The relation $\tilde s_+^2=\tilde s^2+\gamma$ gives
\begin{equation*}
(c+di)^2=(a+bi)^2+\gamma.
\end{equation*}
The real and imaginary part of the two sides of the above equation must be equal, which leads to
\begin{equation*}
\begin{cases}
&cd = ab, \\
&c^2-d^2=a^2-b^2+\gamma.
\end{cases}
\end{equation*}
Therefore, 
\begin{equation}\label{ca}
c^2-a^2=d^2-\frac{c^2d^2}{a^2}+\gamma.
\end{equation}
Assume Re$(\tilde s_+)<\delta$, that is $0\leq c<a$. Then we have $c^2<a^2$, which means that the left hand side of \eqref{ca} is negative. However, by $c^2<a^2$ the right hand side of \eqref{ca} 
\begin{equation*}
d^2-\frac{c^2d^2}{a^2}+\gamma >\gamma\geq 0.
\end{equation*}
This is a contradiction. Therefore, we must have Re$(\tilde s_+)\geq\delta$.
\end{proof}

\section{Proof of Lemma \ref{lemma_stability}}
\label{app_stability}
\begin{proof}
Let\begin{equation}\label{1d_ansatz}
u(t)=e^{st}\phi
\end{equation}
for some complex number $s$ and $\phi$ is a one dimensional grid function with $\|\phi\|_{1D,x}<\infty$. Substituting \eqref{1d_ansatz} to \eqref{1d_semi}, with the notation $\tilde s=sh$ we obtain the eigenvalue problem
\begin{equation}\label{1d_eig}
s^2\phi=\frac{Q}{h^2}\phi.
\end{equation}
We note that \eqref{1d_eig} is in exactly the same form as  \eqref{err_L_1d} with a zero truncation error. As a consequence, the boundary system corresponding to  \eqref{1d_eig} is
\begin{equation}\label{1d_bs}
C(\tilde s)\Sigma=\bold{0},
\end{equation}
where the left--hand side of \eqref{1d_bs} is the same as the left--hand side of \eqref{BS}.

If $C(\tilde s)$ is singular for some $\tilde s$ with Re$(\tilde s)>0$, then \eqref{1d_bs} has a non--trivial solution $\Sigma\neq\bold{0}$. It then follows that \eqref{1d_ansatz} with Re$(s)>0$ is a solution of \eqref{1d_semi}. However, this contradicts stability for the reason outlined in Lemma 12.1.1 in \cite[pp.~378]{Gustafsson2013}. For completeness, we state it below. 

We define a sequence of grids indexed by $n$
\begin{equation*}
x_j^{(n)}=j h_n, \quad h_n=\frac{h}{n}, \quad j=0,1,\cdots, n=1,2,\cdots.
\end{equation*}
The eigenvalue problem reads 
\begin{equation*}
s^2\phi^{(1)}=\frac{Q}{h^2}\phi^{(1)}.
\end{equation*}
We define a sequence of grid functions $f^{(n)}=\phi^{(1)},\quad n=1,2,\cdots$, which satisfies the eigenvalue problem
\begin{equation*}
s^2f^{(n)}=\frac{Q}{h^2}f^{(n)} \Longleftrightarrow n^2 s^2 f^{(n)} = \frac{Q}{h_n^2} f^{(n)}.
\end{equation*}
Therefore, 
\begin{equation*}
u^{(n)}(t) = e^{nst} f^{(n)}
\end{equation*}
are solutions of \eqref{1d_ansatz} and grow arbitrarily fast, i.e. a contradiction to stability. 
\end{proof}

\section{Estimates of $|\Sigma_i|$}
\subsection{Proof of Lemma \ref{antar_C2}}\label{proof_C2}
\begin{proof}
We decompose $\tilde s=i\tilde\xi+\eta h$ in the right--half plane into two parts. The first part is when $\tilde s$ is in the vicinity of the origin, i.e. $0<|\tilde s|\leq\delta$ where $\delta$ is a small constant independent of $h$. Here the range of $\tilde s$ includes a small, but $h$--independent interval of the imaginary axis. For example, we can in particular consider Re$(\tilde s)=\eta h\leq\delta/\sqrt{2}$ and $|\text{Im}(\tilde s)|\leq\delta/\sqrt{2}$. The second part is when $\tilde s$ is away from the origin, Re$(\tilde s)=\eta h$ and $|\text{Im}(\tilde s)|>\delta/\sqrt{2}\sim\mathcal{O}(1)$.

We start with the first case when $\tilde s$ is in the vicinity of the origin. The entries of $C(\tilde s)$ are continuous functions in $\tilde s$. Every entry can be expanded in terms of $\kappa(\tilde s)$ or $\tilde s$, and $\kappa(\tilde s)$ can be further expanded to its Taylor series. $C(\tilde s)$ can be written as 
\begin{equation}
C(\tilde s)=C(0)+\tilde sC'(0)+\frac{\tilde s^2}{2}C''(0)+\cdots,
\end{equation}
with the notation $C'(\tilde s)=dC(\tilde s)/d\tilde s$. An estimate of the type \eqref{sigma_w} can be obtained by using Lemma 3.4 in \cite{Nissen2012}. For completeness, we state this lemma below (notations of norms are changed to be consistent with the notations used in this paper) and explain how to use it thereafter. 

\begin{lemma}\label{Nissenslemma}
(\textit{Lemma 3.4 in \cite{Nissen2012}}) Consider the $n\times n$ linear system $(A+\delta E)x =b$ where $A$ is singular with rank $n-1$. Let $USV^{*}=A$ be the singular value decomposition of $A$. If $(U^{*}EV)_{nn}\neq 0$ then for all sufficiently small $|\delta|$ we get
\begin{equation*}
 \|(A+\delta E)^{-1}\|_{\max}\leq (2\delta(U^{*}EV)_{nn})^{-1}.
\end{equation*}
If, in addition, $b$ is in the column space of $A$, then for all sufficiently small $|\delta|$, we have $\|x\|_{max}\leq c\|b\|_{\max}$. Here $c$ is independent of $\delta$.
\end{lemma}

The first step of using Lemma  \ref{Nissenslemma} is to perform the singular value decomposition (SVD) of $C(0)=\mathcal{USV^*}$. The value of $w$ in \eqref{sigma_w} depends on 1) whether $(U^*C'(0)V)_{nn}$ is equal to zero; 2) whether $\hat T_C$ is in the column space of $C(0)$.
\begin{itemize}
\item When $(U^*C'(0)V)_{nn}\neq 0$: if $T_C$ is in the column space of $C(0)$ then  $w=0$; otherwise $\|C^{-1}(\tilde s)\|_{\max} \sim1/(\eta h)$ and $w=1$.
\item When $(U^*C'(0)V)_{nn}= 0$: we need to take into account the third term in the Taylor series of $C(\tilde s)$, i.e. $C^{''}(0)$, and check whether $(U^*C^{''}(0)V)_{nn}$ is equal to zero. More generally, we obtain the estimate \eqref{sigma_w} with $m=1,2,\cdots,w-1$ if
\begin{equation*}
\left.\left(\mathcal{U}^*\frac{d^mC(\tilde s)}{d\tilde s^m}\right\vert_{\tilde s=0}\mathcal{V}\right)_{nn}=0\text{ and }\left.\left(\mathcal{U}^*\frac{d^wC(\tilde s)}{d\tilde s^w}\right\vert_{\tilde s=0}\mathcal{V}\right)_{nn}\neq0.
\end{equation*}
\end{itemize}
If $w$ in \eqref{sigma_w} is infinite, then $C(\tilde s)$ is singular for some $\tilde s$ with a positive real part. This contradicts to the stability of the numerical scheme. Therefore, with a stable discretization $w$ is always finite.

Next, we consider the case when $\tilde s$ is away from the origin. Assume that $C(\tilde s)$ is singular at $\tilde s=i\tilde\xi$ for some $\tilde\xi$ where $|\tilde s|>\delta\geq 0$ with $\delta$ independent of $h$. Now we may need to use a Puiseux series for $\kappa(\tilde s)$. This leads to a similar expansion of $C(\tilde s)$ as above, but with a non--integer leading order exponent. We round the exponent to its nearest integer, denoted by $\alpha$, towards positive infinity. By the same argument as the case $|\tilde s|\leq\delta$, at  $\tilde s=i\tilde\xi+\eta h$ we have
\begin{equation*}
\|C^{-1}(\tilde s)\|_{\max}\leq \frac{K}{(\eta h)^\alpha}.
\end{equation*}
The non--zero components of $\hat T_C$ are in the form $\partial^{p+2}/\partial x^{p+2} \hat U(0,s)$ as given in \eqref{Tform}. The key is to realize that
\begin{equation*}
\mathcal{L}\left[\frac{\partial^\alpha}{\partial t^\alpha}\frac{\partial^{p+2}U(0,t)}{\partial x^{p+2}}\right]=s^\alpha\frac{\partial^{p+2}}{\partial x^{p+2}}\hat U(0,s),
 \end{equation*}
 where $\mathcal{L}$ is the Laplace transform operator in $t$. Here we have used the compatibility condition between the initial and boundary data.  At  $\tilde s=sh=(i\xi+\eta)h$ we have
 \begin{equation*}
\left\vert\frac{\partial^{p+2}}{\partial x^{p+2}}\hat U(0,s)\right\vert\leq\frac{1}{|s^\alpha|}\left|\mathcal{L}\left[\frac{\partial^\alpha}{\partial t^\alpha}\frac{\partial^{p+2}}{\partial x^{p+2}}U(0,t)\right]\right|. \end{equation*}
Since the singularity is away from the origin, we have
\begin{equation*}
\|\Sigma\|_{\max}\leq \|C^{-1}(\tilde s)\|_{\max} \|\hat T_C\|_{\max} \leq\frac{K}{\eta^\alpha h^\alpha}\frac{1}{|s^\alpha|}\left|\mathcal{L}\left[\frac{\partial^\alpha}{\partial t^\alpha}\frac{\partial^{p+2}}{\partial x^{p+2}}U(0,t)\right]\right|\leq\frac{K}{\eta^{\alpha}\delta^\alpha}\sim\mathcal{O}(1).
 \end{equation*}
This proves that no accuracy loss is caused by the singularity of $C(\tilde s)$ away from the origin. Note that by using this lemma the term 
$\|\mathcal{L}(\partial^\alpha T_C/\partial t^\alpha)\|_{\max}$ is added to the final estimate.
\end{proof}

\section{Estimates of $1/(1-|\kappa_j|^2)$}\label{estimating_one_over_kappa}

\subsection{Proof of Lemma \ref{antar_kappa}}\label{proof_kappa}
\begin{proof}
The characteristic equation takes the general form
\begin{equation}\label{ch_eqn_2l}
\sum_{j=0}^{2l} a_{j}\kappa^j=\tilde s^2\kappa^l,
\end{equation}
where $a_j$ are the coefficients of the standard central finite difference stencils \cite{Fornberg1998}. In the estimate, we only need to consider the $l$ admissible roots and derive estimates for $1/(1-|\kappa_j|^2),\ j=1,2,\cdots,l$. We start with the case when $\tilde s=i\tilde\xi+\eta h$ is in the vicinity of the origin, i.e. $|\tilde s|\leq\delta$ where $\delta$ and $\eta$ are small constants independent of $h$. 

If a root $|\kappa_j(0)|<1$ then a perturbation analysis straightforwardly leads to $|\kappa_j(\tilde s)|<1$ and consequently $1/(1-|\kappa_j(\tilde s)|^2)$ is bounded independently of $h$. However, for a root $|\kappa_j(0)|=1$ a careful derivation is needed to obtain the precise dependence of $1/(1-|\kappa_j(\tilde s)|^2)$ on $h$. 

To proceed, we derive another form of the characteristic equation. Since admissible roots are only related to interior stencils, we consider a uniform grid in one space dimension 
\begin{equation*}
x_j=jh,\quad j=0,\pm 1,\pm 2,\cdots,
\end{equation*}
with grid spacing $h$. We denote $D^{(2l)}$ the $2l^{th}$ order accurate central finite difference operator approximating $\partial^2/\partial x^2$, and apply it to the mode $e^{i\omega x_j}$, where $-\pi<\omega h\leq\pi$. Similarly to the derivation in \cite[pp.~41]{Fornberg1996}, the operator $D^{(2l)}$ acting on the mode $e^{i\omega x_j}$ gives
\begin{equation*}
D^{(2l)}e^{i\omega x_j}=\frac{1}{h^2} f(l,\omega h)e^{i\omega x_j},
\end{equation*}
where
\begin{equation*}
f(l,\omega h)=-\sum_{n=0}^{l-1} \frac{2(n!)^2}{(2n+2)!}\left(4\sin^2\frac{\omega h}{2}\right)^{n+1}.
\end{equation*} 
The corresponding characteristic equation is
\begin{equation}\label{f_xi}
f(l,\omega h)=\tilde s^2.
\end{equation}

A root $|\kappa|=1$ at $\tilde s=0$ corresponds to a root $f(l,\omega h)=0$ for $-\pi<\omega h\leq\pi$. It is obvious that $f(l,\omega h)=0$ has a double root $\omega=0$. In addition, $f(l,\omega h)$ is a monotone decreasing function in $l$ so the only roots of $f(l,\omega h)=0$ are the double root $\omega=0$, corresponding to a double root $\kappa=1$ of the characteristic equation \eqref{ch_eqn_2l} with $\tilde s=0$. We therefore factorize \eqref{ch_eqn_2l}  to 
\begin{equation}\label{ch_factorized}
(\kappa-1)^2 P(\kappa)=0,
\end{equation}
where $P(1)\neq 0$. When $\tilde s=i\tilde\xi+\eta h$, the double root is perturbed to two single roots, where precisely one is admissible. Substituting the admissible root $\kappa_\nu=1+\nu$ to \eqref{ch_factorized}, we obtain
 \begin{equation*}
\nu^2 =\tilde s^2\kappa_\nu^l / P(\kappa_\nu).
\end{equation*}
The value of $\kappa_\nu^l / P(\kappa_\nu)$ to the leading order is real and $\mathcal{O}(1)$. Therefore, there exists constant $K_1$ and $K_2$ independent of $h$ such that 
\begin{equation*}
|\text{Re}(\nu)|\geq K_1\eta h \text{ and } |\text{Im}(\nu)|\geq K_2|\tilde\xi|.
\end{equation*} 
Note that the admissibility condition leads to Re$(\nu)<0$. We now have the estimate 
\begin{equation*}
\frac{1}{1-|\kappa_\nu(\tilde s)|^2}=\frac{1}{1-|1+\nu|^2}=\frac{1}{-2\text{Re}(\nu)-\text{Re}(\nu)^2-\text{Im}(\nu)^2}\leq \frac{K}{\eta h},  
\end{equation*}
for $\tilde s$ in a vicinity of the origin. 

Next, we consider the case when $\tilde s=i\tilde\xi+\eta h$ is away from the origin. Assume an admissible root $|\kappa_a(i\tilde\xi+\eta h)|<1$ satisfies $|\kappa_a(i\tilde\xi)|=1$. The expansion of $\kappa_a(i\tilde\xi+\eta h)$ around $\tilde s=i\tilde\xi$ leads to
\begin{equation}\label{kappa_w}
\frac{1}{1-|\kappa_a(\tilde s)|^2}\leq\frac{K}{(\eta h)^\beta},
\end{equation}
where $\beta$ is the leading order exponent in its Puiseux series rounded to the nearest integer towards positive infinity. We note that when estimating the error \eqref{L2_e_hat}, $\frac{1}{1-|\kappa_a(\tilde s)|^2}$ is multiplied by $|\sigma_a|^2$, which is computed by solving the boundary system \eqref{BS} and is related to the spatial derivatives of the true solution. We can therefore eliminate the $h$--dependence in \eqref{kappa_w} in the same manner as in Lemma \ref{antar_C2}, which increases the order of temporal derivative of the true solution in the final estimate from $\alpha$ to $\alpha+\beta$. 
\end{proof}

\end{document}